\documentclass[11pt]{amsart}
\usepackage{amssymb,latexsym,amsmath,amsfonts}
\usepackage{supertabular}
\usepackage{enumerate}
\usepackage[mathscr]{eucal}
\usepackage{enumitem}
\usepackage{accents}

\allowdisplaybreaks

\numberwithin{equation}{section}
\theoremstyle{definition}
\newtheorem{definition}{Definition}[section]
\newtheorem{example}[definition]{Example}

\theoremstyle{remark}
\newtheorem{remark}[definition]{Remark}

\theoremstyle{plain}
\newtheorem{proposition}[definition]{Proposition}
\newtheorem{theorem}[definition]{Theorem}
\newtheorem{lemma}[definition]{Lemma}
\newtheorem{result}[definition]{Result}
\newtheorem{corollary}[definition]{Corollary}

\usepackage[usenames]{color}

\definecolor{Red}{rgb}{1,0,0}
\definecolor{Magenta}{rgb}{0.69,0,0.83}
\definecolor{Green}{rgb}{0.117,0.706,0.314}
\definecolor{Grey}{rgb}{0.8,0.8,0.8}

\newcommand{\unitdisk}{\mathbb{D}}
\newcommand{\koba}{\mathsf{k}}
\newcommand{\dkoba}{\kappa}

\newcommand{\kobMetGrM}{\mathfrak{M}}
\newcommand{\lk}{(\lambda, \kappa)}

\newcommand{\iprt}{\mathsf{Im}}

\newcommand{\dtb}[1]{\delta_{#1}}
\newcommand{\clos}[1]{\overline{#1}}
\newcommand{\lraw}{\longrightarrow}

\newcommand*{\defeq}{\mathrel{\vcenter{\baselineskip0.5ex \lineskiplimit0pt \hbox{\scriptsize.}\hbox{\scriptsize.}}}=}

\newcommand{\unitin}{[0, \, 1]}

\newcommand{\bdy}{\partial}
\newcommand{\OM}{\Omega}

\newcommand{\smoo}{\mathcal{C}}

\newcommand{\bcdot}{\boldsymbol{\cdot}}

\newcommand{\C}{\mathbb{C}} 
\newcommand{\R}{\mathbb{R}}

\newcommand{\posint}{\mathbb{Z}_{+}}

\begin{document}

\title[Visibility with respect to the Kobayashi distance]{Notions of Visibility with respect to the \\Kobayashi distance: comparison and applications}

\author{Vikramjeet Singh Chandel, Anwoy Maitra, Amar Deep Sarkar}
\address{Department of Mathematics and Statistics, Indian Institute of Technology
Kanpur, Kanpur 208016, India}
\email{vschandel@iitk.ac.in, abelvikram@gmail.com}
\address{Department of Mathematics and Statistics, Indian Institute of Technology
Kanpur, Kanpur 208016, India}
\email{anwoym@iitk.ac.in, anwoymaitra@gmail.com}

\address{School of Basic Sciences, Indian Institute of Technology Bhubaneswar, Khordha 752050, India}
\email{amar@iitbbs.ac.in, amardeeepsarkar@gmail.com}

\keywords{Kobayashi distance, Kobayashi metric, embedded submanifolds, taut submanifolds, Kobayashi isometry, visibility, Wolff--Denjoy theorem}

\subjclass[2010]{Primary: 32F45; Secondary: 32H02, 32H40, 32H50, 53C23}

\begin{abstract}
In this article, we study 
notions of visibility 
with respect to 
the Kobayashi distance for relatively compact complex submanifolds in 
Euclidean spaces. 
We present a sufficient condition 
for a domain to possess the visibility property relative
to Kobayashi almost-geodesics introduced 
by Bharali--Zimmer (we call this simply {\em the visibility property}).
As an application, we produce new classes of domains having this kind of
visibility. Next, we introduce and study the notion of 
{\em visibility subspaces} of relatively compact complex
submanifolds.  Using this notion, we generalize to such submanifolds 
a recent result of Bracci--Nikolov--Thomas. The utility of this generalization 
is demonstrated by proving a theorem on the continuous extension of Kobayashi 
isometries. Finally, we prove a Wolff--Denjoy-type theorem 
that is a generalization of recent results of this kind by 
Bharali--Zimmer and Bharali--Maitra and that, owing to the new classes of domains
mentioned, is a proper generalization. Along the way, we note that 
what is needed for the proof of this sort of theorem to work is a 
form of visibility that seems to be intermediate between what we are calling
{\em visibility} and visibility with respect to ordinary Kobayashi 
geodesics.
\end{abstract}
\maketitle

\vspace{-1.0cm}
\section{Introduction and statement of main results}\label{S:intro}
For a hyperbolic complex manifold $M$, the Kobayashi distance 
$\koba_{M}$ and the Kobayashi--Royden pseudometric $\kappa_{M}$ 
encapsulate many complex-geometric and function-theoretic properties 
of $M$. Recall that $\koba_{M}$ is the integrated form of $\kappa_M$ 
(see Result~\ref{Res:C1GeodesicInf}).
This article 
explores a particular aspect of the Kobayashi distance, namely, notions of 
visibility with respect to real geodesics and almost-geodesics
with respect to the Kobayashi distance (see Definition~\ref{Def:lk_geodesic} in Section~\ref{S:prelims}). 
The two notions of visibility that we shall focus on
originated in the articles \cite{Zimmer}, \cite{Bharali_Zimmer}, 
\cite{Bharali_Maitra} and \cite{Bracci_Nikolov_Thomas},
where these properties were studied for bounded domains in complex 
Euclidean spaces. The motivation for the above two notions is to capture the 
negative-curvature-type behaviour of the 
Kobayashi distance; see the introductions in \cite{Bharali_Zimmer, 
Bharali_Maitra} for a detailed discussion regarding this.
In this article, we shall focus on the analogous properties on 
bounded, connected, embedded complex submanifolds of 
$\C^d$. Not only is this interesting in its own right, but we shall 
see that it leads to some useful applications.
\smallskip 

In what follows, $M$ will always denote a bounded, connected, embedded complex 
submanifold of $\C^d$. 
We shall use $\bdy M$ to denote $\overline{M}\setminus M$, the
boundary of $M$ calculated with respect to $\overline{M}$, 
which is a compact subset of $\C^d$. For $z \in M$, we shall use 
$\dtb{M}(z)$ to denote the Euclidean distance between $z$ and $\bdy M$.
We now introduce the aforementioned notions of visibility for $M$ 
with respect to the Kobayashi distance $\koba_M$.

\begin{definition} \label{def:BZ-vis_submani_reg}
	Let $M$ be as above. Fix $\lambda \geq 1, \kappa> 0$. We say that $M$ {\em has 
		the visibility property with respect to 
		$(\lambda,\kappa)$-almost-geodesics} or that $M$ {\em is a 
		$\lk$-visibility submanifold} if the following two properties 
	hold true.
	\begin{itemize}[leftmargin=14pt]
		\item Any two distinct points of $M$ can be joined by a $(\lambda, \kappa)$-almost-geodesic.
		\smallskip
		
		\item For every pair of
		points $p\neq q \in \bdy M$, there exist $\C^d$-neighbourhoods
		$V$ and $W$ of $p$ and $q$, respectively, and a compact subset $K$ of
		$M$ such that $\clos{V} \cap \clos{W} = \emptyset$ and such that 
		every $(\lambda,\kappa)$-almost-geodesic in $M$ with initial point in
		$V$ and terminal point in $W$ intersects $K$. 
	\end{itemize}
	If $M$ is a {\em domain} $\OM \subset \C^d$ that has the visibility
	property with respect to $(\lambda,\kappa)$-almost-geodesics, then we 
	say that $M$ {\em is a $\lk$-visibility domain}.
	If $M$, as above, has the visibility property with respect to 
	$\lk$-almost-geodesics for {\em every} 
	$\lambda\geq 1$ and $\kappa> 0$, then we shall say that $M$ 
	{\em is a visibility submanifold}. If $M$ has the visibility property with
	respect to $(1,\kappa)$-almost-geodesics for all $\kappa>0$, then we shall say
	that $M$ {\em is a weak visibility submanifold}. Finally,
	we shall say that $M$ {\em is a geodesic visibility submanifold} if $(M, 
	\koba_M)$ is a complete distance space and $M$ satisfies the second property 
	above in the definition of  $\lk$-visibility submanifolds with 
	``$\lk$-almost-geodesic'' replaced by ``real Kobayashi geodesic''.
\end{definition}

Given $M$ as above, a $\lambda\geq 1$ and a $\kappa>0$, it is not 
clear whether, given two distinct points in $M$, there is a 
$(\lambda, \kappa)$-almost-geodesic joining them. When $M=\OM$, a bounded domain 
in $\C^d$, this was proved by Bharali--Zimmer 
\cite[Proposition~4.4]{Bharali_Zimmer}. That this is the case for a general $M$, 
as above, is the content of Theorem~\ref{T:exist_1kappa} in 
Section~\ref{S:prelims}. In the case where $(M, \koba_M)$ is complete, any two 
points in $M$ can be joined by a real geodesic (see Remark~\ref{Rm:HRT}). 
Therefore, the definitions above are not vacuous.
We also mention that for bounded {\em domains} the concept of visibility 
has been studied in the articles \cite{Bharali_Zimmer}, \cite{Bharali_Maitra}.
The concept of geodesic visibility in the context of (bounded) domains $\OM$ for 
which $(\Omega, \koba_\Omega)$ is complete has been studied in the recent 
article \cite{Bracci_Nikolov_Thomas}.
\smallskip

We now turn our attention to visibility domains.
In \cite{Bharali_Zimmer}, it was shown that a large class of domains,
called {\em Goldilocks domains}, possesses the visibility 
property. Since we shall refer to them several times in 
this work, let us introduce them here. For this, we give the following 
definition. Given $M$ as above and given an
open set $U\subset\C^d$ such that $U\cap \bdy M\neq\emptyset$, we define
\begin{equation} {\label{E:kobMetGr}}
\forall\,r>0,\; \kobMetGrM_{M,\,U}(r) \defeq 
\sup\Bigg\{ \frac{1}{\dkoba_{M}(z;v)} 
\mid z\in U\cap M : \delta_M(z) \leqslant r \text{ and }  v\in T_z^{(1,0)}M : 
\|v\|=1 \Bigg\},
\end{equation} 
where $\|\bcdot\|$ denotes the Euclidean norm and $T^{(1,0)}_zM$ denotes the
complex tangent space to $M$ at $z$. We abbreviate $\kobMetGrM_{M,\,\C^d}$ to
$\kobMetGrM_{M}$. Note that, in particular, we can take $M$ to be a bounded
domain in $\C^d$. 
\begin{definition}\label{D:Goldilocks_domain}
	A bounded domain $\OM \subset \C^d$ is called a \emph{Goldilocks domain} if
	\begin{enumerate}
		\item for some (hence any) $\epsilon >0$ we have
		\[
		\int_0^\epsilon \frac{1}{r} \kobMetGrM_\Omega\left(r\right) dr < \infty, \text{ and}
		\]
		\item for each $z_0 \in \OM$, there exist constants $C, \alpha > 0$ (that depend on $z_0$) such that 
		\begin{equation*}\label{E:Goldilocks_koba-distance_bound}
			\koba_{\OM}(z_0, z) \leqslant C + \alpha \log\frac{1}{\delta_{\OM}(z)} \quad \forall z\in \OM.
		\end{equation*}
	\end{enumerate}
\end{definition}
\noindent Examples of Goldilocks domains are given in Section~2 of 
\cite{Bharali_Zimmer}.
In particular, due to a result of S.\,Cho \cite{Cho}, every bounded 
smooth pseudoconvex domain of finite D'Angelo type is a Goldilocks 
domain. 
Later, Bharali--Maitra in \cite{Bharali_Maitra} gave another 
criterion more permissive than
the one in Definition~\ref{D:Goldilocks_domain} for domains to possess the 
visibility property. Using this new criterion, they also constructed domains, 
which they called {\em caltrops}, that possess the visibility property but are not 
Goldilocks domains (more precisely, they do not satisfy condition~(2) 
in Definition~\ref{D:Goldilocks_domain}). Based on the {\bf new}
understanding that the proof that a domain possesses the visibility property
{\em can be localized to the boundary}, 
we present a sufficient condition more permissive
than the one given in \cite[Theorem~1.5]{Bharali_Maitra} for a domain to possess
the visibility property. (Some time after this paper was announced, results were
published (see \cite{Sarkar,Nik_Okt_Tho}) that have substantiated the assertion 
that visibility itself is a local
property of the boundary.) 
\begin{theorem}[Extended Visibility Lemma]\label{T:EVL}
	Let $\Omega$ be a bounded domain in $ \C^d $. Let $ E \subset \partial\Omega $ be a closed set such that
	for every $ p \neq q \in \partial \Omega $, there exist $ p'\in \partial \Omega $ and $ r > 0 $ satisfying
	\begin{itemize}[leftmargin=18pt]
		\item[$a)$] $ p \in B(p', r) $ and $ q \in \partial \Omega \setminus \overline{B(p', r)}$;
		\smallskip
		
		\item[$b)$] $ E \cap \partial B(p', r) = \emptyset $. 
	\end{itemize}
	Further, assume that for every $ q' \in \partial \OM \setminus E $ there exist a neighbourhood $ U $ of $ q' $, a point $z_0\in \OM$
	and a $ \smoo^1 $-smooth strictly increasing function $ f : (0, \infty) \longrightarrow \R $, with $f(t) \to \infty$ as $t\to\infty$,
	such that
	\begin{itemize}[leftmargin=18pt]
		\item[$1)$] for all $ z \in\OM\cap U$, $ \koba_\OM(z_0, z) \leq f(1/\delta_{\OM}(z)) $; 
		\item[$2)$] $ \kobMetGrM_{\OM,\,U}(r) \to 0 $ as $ r \to 0 $, and
		\item[$3)$] there exists $ r_0 > 0 $ such that
		\[
		\int_{0}^{r_0} \frac {\kobMetGrM_{\OM,\,U}(r)}{r^2} f^{\prime}\left(\frac{1}{r}\right)dr < \infty.
		\]
	\end{itemize} 
	Then $\OM$ is a visibility domain.
\end{theorem}
\noindent Here, $B(p',r)$ denotes the open Euclidean ball of radius 
$r$ centred at $p'$ in $\C^d$.
\smallskip

As an application of the above theorem, we prove the following corollary.
\begin{corollary}\label{C:finitetypeexceptsmallset}
	Let $\OM $ be a bounded domain in $ \C^d$. Suppose 
	there exists a compact subset $E \subset \partial\OM $ such that $E_a $, the 
	set of limit points of $E$, is a finite set and such that every $ p \in 
	\partial \OM \setminus E$ is a smooth pseudoconvex boundary point of finite 
	D'Angelo type. Then $\OM $ is a visibility domain.
\end{corollary}
	
\begin{remark}
	A recent result of Bracci--Nikolov--Thomas says that every 
	bounded convex domain with $\mathcal{C}^{\infty}$-boundary is a 
	geodesic visibility domain if all except finitely many boundary points 
	of the domain are of finite D'Angelo type; see 
	\cite[Theorem~1.1]{Bracci_Nikolov_Thomas}. 
	This result can 
	be deduced from 
	Corollary~\ref{C:finitetypeexceptsmallset}. To 
	see this, 
	observe that, by the above corollary, any such 
	domain is a visibility domain, hence a weak visibility domain. 
	Since bounded convex domains are complete, it 
	follows from Corollary~\ref{C:BNTandkappaposit} (see 
	subsection~\ref{SS:bnt_vis_submani}) that any such 
	domain is a geodesic visibility domain.
\end{remark}

\begin{remark}
	Some time after this paper was announced, Bharali--Zimmer 
	generalized
	the above theorem (see \cite[Theorem~1.4]{Bharali_Zimmer_unb}) by 
	removing the assumption of the boundedness of the domain and by allowing
	$E$ to be any closed, totally disconnected subset of the boundary.
\end{remark}
We turn our attention to geodesic visibility. We first remark that it is 
easy to come up with examples where the complete 
distance space $(M, \koba_M)$ does not have geodesic visibility 
although there are subspaces of $M$ that may have this property. 
This possibility motivated us to introduce the following definitions.
Before giving them, we clarify that, given a curve $\gamma$, 
we will 
use the symbol $\gamma$ itself to denote $\mathsf{ran}(\gamma)$. However, 
if there is any danger of confusion, we will 
use the unambiguous notation $\mathsf{ran}(\gamma)$.
We now give

\begin{definition}\label{Def:geodspace}
	Let $M$ be a bounded, connected, embedded complex submanifold of $\C^d$. 
	A subset $S$ of $M$ will be called 
	a {\em geodesic subspace} if the following two conditions are satisfied.
	\begin{itemize}[leftmargin=14pt]
		\item The distance space $(S,\koba_M|_{S \times S})$ is Cauchy-complete.
		\smallskip
		
		\item For any two distinct points in $S$, there 
		exists a geodesic of the 
		space $(M,\koba_M)$ that passes through those points and that 
		is contained in $S$.
	\end{itemize}
\end{definition}

\begin{definition}\label{Def:vissub}
A geodesic subspace $S$ is called a {\em visibility subspace} of $M$ 
if for any $p\neq q \in\bdy_a S:=\overline{S} \setminus S$, there 
exist $\C^d$-neighbourhoods $U$
and $V$ of $p$ and $q$, respectively, and a compact subset $K$ of $S$
such that $\overline{U} \cap \overline{V} = \emptyset$ and such
that, for every $\koba_M$-geodesic $\gamma$ in $S$ with initial point
in $U \cap S$ and terminal point in $V \cap S$, $\mathsf{ran}(\gamma)
\cap K \neq \emptyset$.
\end{definition}

\noindent 
In the case $S=M$, $S$ being a visibility subspace is 
equivalent to $S$ being a geodesic visibility submanifold.
Denoting the open unit disk in $\C$ by $\unitdisk$,
we ask the reader to note that 
$(\unitdisk^n,\koba_{\unitdisk^n}),\,n\geq 2$, is not a 
geodesic visibility domain (this is easy to see). It is known (see
\cite{Heath_Suffridge}) that all the one-dimensional retracts of
$\unitdisk^n$, $n \geq 2$, are given by 
$V_f \defeq \{(z,f(z)):z\in\unitdisk\}$, 
where $f=(f_1,\dots,f_{n-1}):\unitdisk\longrightarrow\unitdisk^{n-1}$ 
is a holomorphic map. It is not too difficult to show that $V_f$ is a
visibility subspace of $\unitdisk^n$ if each $f_j$ extends 
continuously to $\overline{\unitdisk}$. We sketch an argument 
showing this at the beginning of 
Section~\ref{S:Ext_Kob_Iso}. The converse is also true,
namely, if $V_f$ is a visibility subspace of $\unitdisk^n$, each 
$f_j$ extends continuously to $\overline{\unitdisk}$. One can see 
this by noting that $V_f$ is actually the image of a complex geodesic
in $\unitdisk^n$, namely of $z \mapsto (z,f(z)) : \unitdisk \lraw
\unitdisk^n$, and then using Theorem~\ref{th:cont_ext_koba_isom} 
below.
\smallskip

We now present a result regarding the visibility of geodesic subspaces. This result is a
generalization of Theorem~3.3 in \cite{Bracci_Nikolov_Thomas} to the context of geodesic subspaces. This 
theorem might seem overly abstract, but its utility will become apparent when we prove 
Theorem~\ref{th:cont_ext_koba_isom} and its corollaries.

\begin{theorem} \label{th:ext_of_id_vis}
	Let $M$ be a bounded, connected, embedded complex submanifold of $\C^d$. 
	Let $S\subset M$ be a geodesic subspace of $M$ such that $\big(S,\koba_M|_{S\times S}\big)$ is 
	Gromov hyperbolic.  
	Then $S$ is a visibility subspace of $M$ if and only if the identity map 
	$\mathsf{id}_S : S \longrightarrow S$ extends to a continuous surjective map
	$\widehat{\mathsf{id}}_S : \overline{S}^G \lraw\overline{S}$, where
	$\overline{S}^G$ denotes the Gromov compactification of $S$ with
	respect to $\koba_M|_{ S\times S}$. Moreover, this extended map is a
	homeomorphism if and only if $S$ has no geodesic loops in $\overline{S}$.
\end{theorem}

\noindent We refer the reader to \cite[Part~III, Chapter~3]{Bridson_Haefliger} 
for the definition of the Gromov compactification of a  
proper, geodesic distance
space that is Gromov hyperbolic. (Also see 
\cite[Section~3]{Bracci_Nikolov_Thomas} for a quick introduction to 
the same when the distance is the Kobayashi distance.) We also refer 
the reader to Definition~\ref{dfn:geod_loop} in Section~\ref{S:Ext_Kob_Iso} for
the definition of a geodesic loop.
\smallskip

Using the above theorem, we prove a
result concerning the continuous extension of Kobayashi isometries.

\begin{theorem}\label{th:cont_ext_koba_isom}
	Suppose that $M \subset \C^m$ and $N \subset \C^n$ are bounded,
	connected, embedded complex submanifolds and that $(M,\koba_M)$ is 
	a complete 
	Gromov hyperbolic space.
	Let $f:M \lraw N$ be an isometric embedding with respect to the
	Kobayashi distances and suppose that $S \defeq f(M)$ (which is 
	easily seen to be a
	geodesic subspace of $N$) is a visibility subspace of $N$. Then $f$
	extends to a continuous map $\widehat{f} : \overline{M}^G \lraw 
	\overline{N}$, where $\overline{M}^G$ denotes the Gromov 
	compactification of $(M,\koba_M)$.
	Further, if $S$ has no geodesic loops in $\overline{S}$, then 
	$\widehat{f}$ is a homeomorphism from $\overline{M}^G$ to 
	$\overline{S}$. 
\end{theorem}

To illustrate the use of this theorem,
we provide the following corollary, which partly
generalizes \cite[Theorem~1.3]{Maitra}. But first
a few words about two concepts that occur in the statement of the 
corollary below. The first one is that of the 
$\smoo^{1,{\rm Dini}}$-smoothness of
the boundary of a domain in $\C^d$. Since we are not going to make 
explicit use of $\smoo^{1,{\rm Dini}}$-smoothness, we shall not define
it here. We direct the reader to \cite{Nikolov_Lyubomir} for the definition.
We note only that if a domain has
$\smoo^{1,\alpha}$-smooth boundary, where $\alpha>0$ is arbitrary, 
then it automatically has $\smoo^{1,{\rm Dini}}$-smooth boundary. In 
particular, all domains with $\smoo^k$-smooth boundary, where $k \geq 
2$, have $\smoo^{1,{\rm Dini}}$-smooth boundary. The second concept 
is that of 
$\C$-strict convexity, which we shall define in detail. Before
giving the definition, we recall that, given a domain $\OM \subset 
\C^d$ and a
$\smoo^1$-smooth boundary point $p$ of $\OM$, we can consider the 
(real) tangent space $T_p(\bdy\OM)$, where we view it 
{\em extrinsically} (i.e., as a real hyperplane
in $\C^d$), and we can also consider {\em the complex tangent space 
to $\bdy\OM$ at} $p$, given by $H_p(\bdy\OM) \defeq T_p(\bdy\OM) \cap 
i\big(T_p(\bdy\OM)\big)$.

\begin{definition} \label{Def:C-str-cvx_smoo}
Given a convex domain $\OM \subset \C^d$ with $\smoo^1$-smooth boundary, a 
boundary point $p$ of $\OM$ is said to be {\em a $\C$-strictly convex boundary
point} if $\big(p + H_p(\bdy\OM)\big) \cap \overline{\OM} = \{p\}$, where 
$H_p(\bdy\OM)$ is the complex tangent space to $\bdy\OM$ at $p$. 
\end{definition}

\begin{corollary}\label{Cor:ext_cg_dinismooth_strcvx}
Let $\OM$ be a bounded convex domain in $\C^d$. Suppose that there
exists a compact subset $S \subset \bdy \OM$ such that $S_a$, the set
of limit points of $S$, is a finite set, and such that every 
$p \in \bdy \OM \setminus S$ is a $\smoo^{1,{\rm Dini}}$-smooth
$\C$-strictly convex boundary point. Let $f: \unitdisk \lraw \OM$ be a
complex geodesic. Then $f$ extends continuously to 
$\overline{\unitdisk}$.
\end{corollary}

\noindent The proof of the above corollary is given at the end of
Section~\ref{S:Ext_Kob_Iso}.
\smallskip

Finally, we move to a topic in the theory of iterations of a 
holomorphic self-map, where the visibility property turns out to imply
some interesting consequences. 
In this direction, we begin with the following famous result due,
independently, to Denjoy and Wolff \cite{Denjoy, Wolff}.
\begin{result}[Denjoy, Wolff]
Let $f : \unitdisk\lraw\unitdisk$ be a holomorphic map. Either $f$ has a unique fixed
point in $\unitdisk$ or there exists a point $p\in \bdy\unitdisk$ such
that $f^{ n}(z)\to p$ as $n\to \infty$ for each $z\in \unitdisk$.
In the latter case, convergence is uniform on compact subsets of $\unitdisk$.
\end{result}
In several works, see e.g. \cite{Bharali_Zimmer, Bharali_Maitra, Zimmer}, it is noted that similar phenomena 
regarding the iteration of a holomorphic self-map of a bounded 
domain\,---\,as exhibited above in the case of 
$\unitdisk$\,---\,could be understood and explained by appealing to 
the visibility property of the Kobayashi distance. 
Using this connection, Bharali--Maitra proved two Wolff--Denjoy-type 
theorems \cite[Theorems~1.8 and 1.9]{Bharali_Maitra} 
for taut domains possessing the visibility property. Our 
next result improves upon \cite[Theorem~1.8]{Bharali_Maitra} 
in two ways: our result is a Wolff--Denjoy-type 
theorem for bounded, taut {\em submanifolds} of $\C^d$ on 
which only {\bf weak} visibility is assumed.
We now present the result.
\begin{theorem} \label{th:WD_submani}
	Suppose that $M$ is a bounded, connected, embedded complex submanifold of 
	$\C^d$ that satisfies the weak visibility property and is taut. Let 
	$F:M\lraw M$ be a holomorphic map. Then exactly one of the following holds:
	\begin{enumerate}
		\item For every $z \in M$, $\{F^\nu(z) : \nu \in \posint\}$ is a
		relatively compact subset of $M$;
		\smallskip
				
		\item there exists $\xi \in \bdy M$ such that, for every $z \in M$,
		$\lim_{\nu \to \infty} F^\nu(z) = \xi$, this convergence being
		uniform on the compact subsets of $M$. 
	\end{enumerate}
\end{theorem}
\begin{remark}\label{Rm:novelty_proof}
The method of our proof 
follows very closely that 
of \cite[Theorem~1.8]{Bharali_Maitra}. A crucial tool that was employed in 
the latter proof
(see \cite[Theorem~4.3]{Bharali_Maitra}) was a 
consequence of the fact that $\lim_{r\to 0}\kobMetGrM_{\OM}(r)=0$
for any taut visibility domain $\OM$. We prove an 
analogous result (Theorem~\ref{th:visib_taut_KobMetGr_0} in
Section~\ref{S:WD}) for an $M$ as above that is taut and that is
only assumed to have the {\em weak} visibility property.
We
emphasize that our theorem is particularly useful where the given submanifold 
is not known to possess the visibility property, but {\em is} known to possess the 
{\em weak} visibility property;
see e.g. Example~\ref{SS:BNTbutnotgoldilocks} in 
Section~\ref{S:twoexam}.
\end{remark}
We now present the plan of this paper.
In Section~\ref{S:prelims}, we present preliminary
material relating to the Kobayashi distance and metric on relatively
compact complex submanifolds of $\C^d$
and prove that
almost-geodesics joining arbitrary pairs of points exist on any such
manifold. In 
Section~\ref{S:local_cons_EVL}, we present the proof of
Theorem~\ref{T:EVL} and then compare the various notions of 
visibility that appear in this paper. In Section~\ref{S:twoexam}, we
present two examples of domains that are not Goldilocks domains but
nevertheless possess some form of visibility. In 
Section~\ref{S:Ext_Kob_Iso}, we study basic properties of
geodesic and visibility subspaces and then prove Theorems~\ref{th:ext_of_id_vis} 
and \ref{th:cont_ext_koba_isom}. Finally, in
Section~\ref{S:WD}, we prove Theorem~\ref{th:WD_submani} after proving
some preliminary results, which are interesting in their own
right.  
  
\section{Preliminaries} \label{S:prelims}
In this section, we shall show the existence of $\lk$-almost-geodesics
(with respect to the Kobayashi distance) 
for a  bounded, connected, embedded complex submanifold $M$ of $ \C^d$.
Before we begin, we recall that the 
Kobayashi pseudometric $\dkoba_M$ is upper semicontinuous. Therefore,
for any continuous mapping $\gamma : I \lraw T^{(1,0)}M$, 
where $I \subset \R$ 
is an interval and $T^{(1,0)}M$ denotes the complex tangent bundle of $M$,
\[
\int_I \dkoba_M(\gamma(t)) dt 
\]
makes sense (it may be $\infty$). Therefore, if we have an embedded
complex submanifold $M$ of $\C^d$ and we have a piecewise $\smoo^1$ 
curve $\gamma : [a,b] \lraw M$, where $a,b \in \R$, $a \neq b$, then
\begin{equation} \label{eq:GenKobaIntXpr}
\int_{a}^{b} \dkoba_M(\gamma(t),\gamma'(t)) dt<\infty.
\end{equation}
It also follows easily that if $\gamma : 
[a,b] \lraw M$ is an absolutely continuous curve, then the
Lebesgue integral of the function $t \mapsto 
\dkoba_M(\gamma(t),\gamma'(t)) : [a,b] \lraw [0,\infty)$ is defined (a 
priori, it could be $\infty$). Therefore, given $\gamma : [a,b] \lraw
 M$, an absolutely continuous curve, we let $l_M(\gamma)$ denote the integral
\eqref{eq:GenKobaIntXpr}, {\em the length of $\gamma$ calculated using
the Kobayashi pseudometric}. In what follows, the following result is relevant.
\begin{result}\label{Res:C1GeodesicInf}
Let $M$ be a connected, embedded complex submanifold  of $\C^d$.\\ 
\noindent{{\normalfont (1)}} \cite[Theorem~1]{Royden} For any $z, w \in M$, we have
\[
\koba_M(z, w) = \inf \left\{ l_M(\gamma) \mid \gamma: [a, b] 
\longrightarrow M \,\, \text{ is piecewise $\mathcal{C}^1$, with 
$\gamma(a) = z$ and $\gamma(b) = w$}\right\}.
\]
We can also take $\gamma$ to be $\mathcal{C}^1$ above.
\smallskip  

\noindent{{\normalfont (2)}} \cite[Theorem~3.1]{VenturiniManifold} For 
any $z, w \in M$, we have
\begin{align*}  
\koba_M(z,w) = &\inf \big\{ l_M(\gamma) \mid \gamma: [a, b] \lraw M \,\, \text{is absolutely continuous, with } \gamma(a) = z \text{ and}\\
&\gamma(b) = w \big\}.
\end{align*} 
\end{result}

We now present a result that is at the heart of the main result of this section, namely the existence of $\lk$-almost-geodesics. 
\begin{proposition}\label{P:PreliForExistenceLamdaKappa}
Let $ M $ be a bounded, connected, embedded complex submanifold of 
$\C^d$. Then the following hold.
\begin{enumerate}
	\item There exists $ c > 0 $ such that 
	\[
	c\, \Vert X \Vert \leq \kappa_M(z, X)
	\]
	for all $ z \in M $ and $ X \in T^{(1, 0)}_zM $.
	\smallskip
	
	\item For any compact set $ K \subset M $, there exists a constant $ C_1 = C_1(K) > 0 $ so that
	\[
          \kappa_M(z, X) \leq 	C_1\, \Vert X \Vert
	\]
	for all $ z \in K $ and $ X \in T^{(1, 0)}_zM $.
\end{enumerate}
\end{proposition}
\begin{remark}\label{Rm:analogKoba}
Part~(1) together with Result~\ref{Res:C1GeodesicInf} implies 
immediately
that $c\, \Vert z - w \Vert \leq \koba_M(z, w)$ for all $ z, w \in M $. 
Similarly, working with Part~(2) above, one can show that
for any compact set $ K \subset M $, there exists a constant $ C_2 = C_2(K) > 0 $ 
such that $\koba_M(z, w) \leq C_2\, \Vert z - w \Vert$
for all $ z, w \in K $. 
\end{remark}
 
\noindent Before we prove the proposition above, we shall state a 
result that will be used in the proof of the proposition.
\begin{result}[\cite{Royden}, Proposition~2]\label{L:PoydiskBound}
	Let $\mathscr{M}$ be a complex manifold of dimension $ n $. If a compact set $ K \subset\mathscr{M}$ is contained in a coordinate polydisk, then there exists a constant $C = C(K)$ 
	such that
	\[
	\dkoba_{\mathscr{M}}(z, X) \leq C \Vert X \Vert 
	\]
	for all  $z \in K$, $X \in T^{(1, 0)}_{z}\mathscr{M}$.
\end{result} 
A co-ordinate polydisk in $\mathscr{M}$ is essentially a co-ordinate chart $(\psi, U, 
\psi(U))$, $U\subset \mathscr{M}$ being open,  such that $\psi(U)$ is a polydisk in $\C^n$. 

\begin{proof}[Proof of Proposition~\ref{P:PreliForExistenceLamdaKappa}]
The proof of part $(1)$ of the proposition is closely analogous to that of part 
$(1)$ of \cite[Proposition~3.5]{Bharali_Zimmer}, so we omit it here. 
\medskip

To establish part~$(2)$, we choose, for each $z\in K$, a coordinate 
polydisk $ U_z $ of $ M $ centred at $z$.
Let $ U_z' \subset U_z $ be another coordinate polydisk centred
at $ z $ that is relatively compact in $ U_z $ for all $ z \in K $. Since $ K $ is compact,
there are finitely many elements of $\{ U_z' : z \in K\}$ that cover $K$. Let $ \{U_{z_i}'\}_{i = 1}^{k} $ be a finite cover, for some $ k \in \posint$. Then, since $ \overline U'_{z_i} $ is a compact subset of $ U_{z_i} $ for all $ i = 1, \dots, k $, by Result~\ref{L:PoydiskBound}, 
\[
\dkoba_M(z, X) \leq C_i\Vert X \Vert
\] 
for all $ z \in U'_{z_i} $ and $ X \in T^{(1, 0)}_{z}M $, where $C_i$
is a constant depending on the compact set $ \overline U'_{z_i} $. 
Set $C 
= C(K) := \max\{C_i: i =1, \dots, k \}  $. Then
\[
\dkoba_M(z, X) \leq C \Vert X \Vert
\]
for all $ z \in K \subset \bigcup U'_{z_i} $ and $  X \in T^{(1, 0)}_{z}M  $. This shows that part (2) is true.
\end{proof}

\begin{definition}\label{Def:lk_geodesic}
Let $M\subset\C^d$ be as before and let $I \subset \R$ be an interval. A {\em real Kobayashi geodesic}
is a map $\sigma:I\lraw M$ that is an isometric embedding, i.e., for any $s,t\in I$, we have
\[
|s-t|=\koba_M(\sigma(s), \sigma(t)).
\]

\smallskip

For $\lambda \geq 1$ and $\kappa \geq 0$, a curve $\sigma: I 
\longrightarrow M$ is called a {\em $ (\lambda, \kappa) 
$-almost-geodesic} if 
\begin{enumerate}
\item for all $s, t \in I$, $(1/\lambda)|t -s| - \kappa \leq \koba_M\big(\sigma(t), 
\sigma(s)\big) \leq \lambda|t -s| + \kappa; \text{ and}$
\vspace{2mm}

\item $\sigma$ is absolutely continuous, so that $\sigma'(t)$ exists 
for almost every $t \in I$, and, for almost every $t \in I$, $\kappa_M(\sigma(t), \sigma^{\prime}(t)) \leq \lambda$.
\end{enumerate}
A curve $\sigma: I \longrightarrow M$ that is only required to satisfy condition~(1) above is called a {\em $(\lambda, \kappa)$-quasi-geodesic}. Note that
such curves are not necessarily continuous.
\end{definition}

\begin{remark}\label{Rm:HRT}
It is a fact that given $M$ as above, the distance space $(M,\koba_M)$ is a {\em length space} that is locally compact.
It is a consequence of the Hopf--Rinow Theorem that if
$(M, \koba_M)$ is (Cauchy-) complete, then given any two points 
$p\neq q\in M$, there is a real geodesic connecting them, i.e., a 
geodesic $\sigma:[0, \koba_M(p,q)]\lraw M$ with $\sigma(0)=p$ and
$\sigma(\koba_M(p,q))=q$. It is also a consequence of that theorem that
when $(M,\koba_M)$ is complete, it is proper, i.e., the closed
$\koba_M$-balls are compact. The reader is referred to \cite[Part~I, Chapter~3]{Bridson_Haefliger} for the definition of length
space and for the statement of the Hopf--Rinow Theorem.
\end{remark}

The following result is the consequence of part $(1)$ of the above proposition. Its proof is exactly the same 
as in \cite[Proposition~4.3]{Bharali_Zimmer}. Therefore we omit it here. 

\begin{proposition}\label{Prop:C_lips_of_almgeod}
Let $M\subset\C^d$ be as before. Then 
for any $\lambda \geq 1$, there exists $C = C(\lambda) > 0$
so that any $(\lambda,\kappa)$-almost-geodesic $\sigma: I 
\longrightarrow M$ is $C$-Lipschitz with respect to the Euclidean 
distance.  
\end{proposition}

We are now ready to present the principal result of this section. 

\begin{theorem}\label{T:exist_1kappa}
Let $M\subset\C^d$ be as before.
For any $ z, w \in M $ and any $ \kappa > 0 $, there 
exists a $ (1, \kappa) $-almost-geodesic $ \sigma : [a, \, b] 
\longrightarrow M $ such that $ \sigma(a) = z $ and $\sigma(b) = w$.	
\end{theorem}
\begin{proof} The proof is an adaptation of the proof of \cite[Proposition~4.4]{Bharali_Zimmer}.
Here we shall only give the main idea of the proof. First, 
by part (1) of Result~\ref{Res:C1GeodesicInf}, there exists a 
$\mathcal{C}^1 $ curve $ \gamma : [0, \, 1] \longrightarrow M $ such 
that $\gamma(0) = z$, $\gamma(1) = w$, and such that 
\[
l_M(\gamma) < \koba_M(z, w) + \kappa.
\]
In addition, we can assume that $ \gamma^{\prime}(t) \neq 0 $ for all $ t \in [0, \, 1] $.
Now, we consider the arc-length function, namely, 
\[
f \defeq t \mapsto \int_0^t \kappa_{M}(\gamma(r),\gamma^{\prime}(r))dr
\; : \; [0,1] \lraw [0,\infty).
\]
Using Proposition~\ref{P:PreliForExistenceLamdaKappa} above, 
we show that $ f $ is a bi-Lipschitz function,
and consequently strictly increasing. Let $ g: [0, l_M(\gamma)] \longrightarrow \unitin $ be the inverse
of $ f $, that is, $ g(f(t)) = t $ for all $ t \in \unitin $. Consider the reparametrization of $\gamma$ defined by 
$ \sigma := \gamma \circ g$. It is not difficult to show that $\sigma$ has unit-speed with respect to $\kappa_M$. This, in particular, 
implies that $ \sigma $ is an $ (1, \kappa) $-almost-geodesic. We refer the reader to \cite[Proposition~4.4]{Bharali_Zimmer} 
for more details. 
\end{proof} 

Note that given $\lambda\geq 1$ and $\kappa>0$, every $(1,\kappa)$-almost-geodesic is a $(\lambda,\kappa)$-almost-geodesic too. Hence
Theorem~\ref{T:exist_1kappa} implies the existence of $(\lambda,\kappa)$-almost-geodesics for any  $\lambda\geq 1$ and $\kappa>0$. 
\smallskip

We end this section with the following
simple result about $(1,\kappa)$-quasi-geodesics that we shall need 
later in the article. This result is a direct consequence of 
the definition of $(1,\kappa)$-quasi-geodesic together with the 
triangle inequality. So we omit the proof.
\begin{result}\label{Res:basic_but_important}
Let  $M\subset \C^d $ be as before. If $\sigma: [a,b]\lraw M$ is a 
$(1,\kappa)$-quasi-geodesic,
then for all $t \in [a,b]$ we have
	\[
	  \koba_{M}(\sigma(a), \sigma(b))\leqslant \koba_{M}(\sigma(a),\sigma(t))
	  + \koba_{M}(\sigma(t), \sigma(b))
	  \leqslant \koba_{M}(\sigma(a),\sigma(b)) + 3\kappa.
	\]
\end{result}

\section{The Extended Visibility Lemma and relations amongst different types of visibility}\label{S:local_cons_EVL}
In this section, we present the proofs of Theorem~\ref{T:EVL} and 
Corollary~\ref{C:finitetypeexceptsmallset}.
As hinted at in the introduction, the proofs of these results demonstrate the 
fact that the proof that a given domain or submanifold possesses the visibility
property is localizable. This realization also motivates us to reconsider geodesic 
visibility and its relation with weak visibility. 
This we present in subsection~\ref{SS:bnt_vis_submani}. Finally, in 
subsection~\ref{SS:comp_BZ_BNT}, we compare visibility and geodesic visibility.

\subsection{The proofs of Theorem~\ref{T:EVL} and Corollary~\ref{C:finitetypeexceptsmallset}}\label{SS:proof_th_EVL}

\begin{proof}[The proof of Theorem~\ref{T:EVL}] Suppose $\OM$ is not a 
visibility domain. Then there exist 
$\lambda\geq 1$ and $\kappa>0$ such that $\Omega$ does not have the visibility 
property with respect to
$(\lambda,\kappa)$-almost-geodesics. This implies that
there exist $ p \neq q \in \partial \OM $, sequences 
$(p_n)_{n \geq 1}$ and $(q_n)_{n \geq 1}$ in $ \OM $ that
converge to $p$ and $q$ respectively, and a sequence 
$(\gamma_n)_{n \geq 1}$  of 
$ (\lambda, \kappa) $-almost-geodesics\,---\,$ \gamma_n: [a_n, b_n] 
\longrightarrow \OM $
with $ \gamma_n(a_n) = p_n $ and $ \gamma_n(b_n) = q_n $ for all $ n \in \posint 
$\,---\,such that $\max_{a_n \leq t \leq b_n}\delta_\OM(\gamma_{n}(t)) \to 0$ as
$n\to\infty$.
By assumption, there exist $ p'\in \partial \OM $ and $ r > 0 $ such that $ p \in 
B(p', r) $, $ q \in \partial\OM \setminus \overline{B(p', r)} $
and $E \cap \partial B(p', r) = \emptyset $. Since $ p_n \to p $ and $ q_n \to q $ 
as $ n \to \infty $, we may assume 
that $p_n \in B(p',r)$ and
$q_n \in  \OM\setminus \overline{B(p',r)}$ for all $n$. Now, as
$ \gamma_n $ is a continuous path from $ p_n $ to $ q_n $, we have 
$\gamma_{n}\big([a_n, b_n]\big)\cap \partial B(p', r)\neq\emptyset$.
Let $ \alpha_n \in (a_n, b_n) $ be such that $ \xi_n \defeq \gamma_{n}(\alpha_n) \in \partial B(p', r) $,
and passing to a subsequence if necessary, we may assume that $ \xi_n \to \xi \in \partial \OM \cap \partial B(p', r) $ as $ n \to \infty $.
Note that $\xi \in \partial\OM \setminus E$; hence, by assumption, 
there exists a neighbourhood $ U $ of $ \xi $
such that conditions $(1)$, $(2)$ and $(3)$ occurring in the statement 
of Theorem~\ref{T:EVL} are 
satisfied. Now observe, since $\kobMetGrM_{\OM,\,V}(r) \leq  
\kobMetGrM_{\OM,\,U}(r) $
for any neighbourhood $ V \subset U $ of $ \xi $, that, 
shrinking $ U $ if necessary, we may assume that
$ \overline{U} \cap (E \cup \{p, q\}) = \emptyset $ and that 
$U$ satisfies the three conditions referred to above.
\smallskip

Let $ \epsilon > 0 $ be such that $ \overline{B(\xi, \epsilon)}\subset U $.
Since $ \xi_n \to \xi $ as $ n \to \infty $, we may assume (without loss of generality) that $\xi_n \in B(\xi,\epsilon)$ for all $n$.
Let
\[
\beta_n \defeq \inf\big\{t \in [\alpha_n, b_n] : \gamma_n(t) \in \partial B(\xi, \epsilon) \big\}.
\]
By definition of $\beta_n$ and the fact that $\partial B(\xi, \epsilon)$ is closed,
we have $\gamma_n(\beta_n)\in\partial B(\xi, \epsilon)$  and $a_n < \alpha_n < \beta_n < b_n $.
For every $n\in\mathbb{Z}_{+}$, define
$ \sigma_n \defeq \gamma_{n}|_{[\alpha_n,\,\beta_n]}: [\alpha_n, \beta_n] \longrightarrow \OM $. 
It is easy to see that $ \sigma_n\big([\alpha_n, \beta_n]\big)\subset 
\overline{B(\xi, \epsilon)}$ for all $ n $.
Note, since $ \sigma_n $ is a restriction of the $ (\lambda, \kappa) 
$-almost-geodesic $ \gamma_n $, 
that $ \sigma_n $ is also a $ (\lambda, \kappa) $-almost-geodesic, for all $ n $. 
Moreover, we have $\max_{\alpha_n \leq t \leq \beta_n} \dtb{\OM}(\sigma_n(t)) \to 
0$ as $n\to\infty$.
We observe that $\sigma_n\big([\alpha_n, \beta_n]\big)\subset \OM\cap U$ for all $n\in\posint$. 
From this point on, we argue exactly as in the proof of Theorem~1.5 in \cite{Bharali_Maitra} (and replace $M_\OM$ by $\kobMetGrM_{\OM,\,U}$ in
the latter proof) to get the result.
\end{proof}

We now present the proof of Corollary~\ref{C:finitetypeexceptsmallset}.

\begin{proof}[The proof of Corollary~\ref{C:finitetypeexceptsmallset}]
	To prove this corollary we shall use the Extended Visibility 
	Lemma. First, we shall show 
	that given $E$ as in the statement of 
	Corollary~\ref{C:finitetypeexceptsmallset} and
	$p\neq q\in\partial \OM$, conditions $(a)$ and $(b)$ in the 
	statement of Theorem~\ref{T:EVL}
	are satisfied. To this end, consider $E_0\defeq E_a\cup\{p, q\}$. 
	Then, owing to the finiteness of $E_0$, there exists an 
	$\epsilon_0$  
	such that $\overline{B(x, \epsilon_0)} \cap \overline{B(x',\epsilon_0)}= \emptyset$ for all $x\neq x'\in E_0$.
	Now define
	\[
	E_1\defeq\big(E\cup\{p,q\}\big) \setminus \big(\cup_{x\in E_a}\overline{B(x, \epsilon_0)}\big).
	\]
	Note that $E_1$ is a finite set disjoint from the compact set 
	$K:=\cup_{x\in E_a}\overline{B(x, \epsilon_0)}$.
	Therefore there exists an $\epsilon_1>0$ such that
	\begin{itemize}[leftmargin=14pt]
	\item $\overline{B(y, \epsilon_1)}\cap K=\emptyset \ \ \forall y\in E_1$ ;
	\smallskip
	
	\item $\overline{B(y, \epsilon_1)}\cap \overline{B(y', \epsilon_1)} = \emptyset \ \ \forall y\neq y'\in E_1$.
	\end{itemize}
	We now consider two cases:
	\smallskip
	
	\noindent{\bf Case 1.} $p\notin K$.
	\smallskip
	
	\noindent In this case if we take $p'= p$ and $r=\epsilon_1$ then 
	both the conditions $(a)$ and $(b)$ in Theorem~\ref{T:EVL} are 
	satisfied.
	\smallskip
	
	\noindent{\bf Case 2.} $p\in K$.
	\smallskip 
	
	\noindent There exists $x_0\in E_a$ such that $p\in\overline{B(x_0, \epsilon_0)}$. Consider the following 
	finite collection of mutually disjoint sets 
	\[
	\mathcal{B}\defeq\big\{B(x, \epsilon_0): x\in E_a\big\}\cup \big\{B(y, \epsilon_1): y\in E_1\big\}.
	\]
	Choose $\epsilon_2$ such that  $\epsilon_2 < 
	\mathsf{dist}(B_1, B_2)/ 4$ for all $ B_1\neq B_2 \in \mathcal{B}$. Then it follows that 
	$\mathcal{C}\defeq\big\{B(x, \epsilon_0+\epsilon_2): x\in E_a\big\}\cup \big\{B(y, \epsilon_1+\epsilon_2): y\in E_1\big\}$
	is a collection of mutually disjoint sets. Now, if we take $p'=x_0$ and $r=\epsilon_0+\epsilon_2$ then both the conditions $(a)$ and $(b)$ 
	 in Theorem~\ref{T:EVL} are satisfied.
	\smallskip
	
Now take an arbitrary point $q' \in \bdy \OM\setminus E$. By
\cite[Theorem~1]{Cho}, there exist a neighbourhood $U$ of $q'$ in 
$\C^d$ and positive numbers $c,\epsilon$ such that
\begin{equation*}
\forall z \in \OM \cap U, \; \forall v \in \C^d, \;
\dkoba_{\OM}(z,v) \geq c\frac{\|v\|}{\dtb{\OM}(z)^{\epsilon}}. 
\end{equation*}
Therefore, for $r>0$ sufficiently small, $\kobMetGrM_{\OM,\,U}(r) 
\leq (1/c)r^{\epsilon}$.
\smallskip
	
It is also a straightforward consequence of 
\cite[Theorem~7]{Nikolov_Lyubomir} that, by shrinking $U$ further if
necessary, we may assume that there exist a point $z_0 \in \OM$ and 
a real number $A$ 
such that, putting $f(x) \defeq A + (1/2) \log(x) \; \forall \, x \in 
(0,\infty)$, we have
\begin{equation*}
\forall z \in \OM \cap U, \; \koba_\OM(z_0,z) \leq 
f\big(1/\dtb{\OM}(z)\big). 
\end{equation*} 
We note that the estimate on $\kobMetGrM_{\OM,\,U}(r)$ derived above
also holds for this possibly smaller $U$. It is now easy to check 
that all the conditions in Theorem~\ref{T:EVL} are satisfied. 
Consequently, invoking Theorem~\ref{T:EVL}, we conclude that $\OM$ is 
a visibility domain.	
\end{proof}
\medskip

\subsection{Weak visibility and geodesic visibility}\label{SS:bnt_vis_submani}
Before we present our first result, we need a definition. Given a distance space 
$(X, d)$ and an arbitrary but fixed point
$o\in X$, the {\em Gromov product} is defined by
\[
 (x|y)_o \defeq (d(x,o)+d(y,o)-d(x,y))/2 \ \ \forall\, x,y\in X.
 \]
We now present 
 \begin{proposition} \label{P:charac_visib_Grom_prod_gen}
	Suppose that $M$ is a bounded, connected, embedded complex 
	submanifold of $\C^d$.	
	\begin{enumerate}
		\item If $M$ has the visibility property with respect to 
		$(1,\kappa)$-almost-geodesics for some $\kappa > 0$, then,
		for every $p,q \in \bdy M$ with $p \neq q$, $\limsup_{(x,y)\to (p,q)}
		(x|y)_o < \infty$.
		\smallskip
		
		\item If, for every $p, q \in \bdy M$ with $p \neq q$,
		$\limsup_{(x,y)\to (p,q)} (x|y)_o < \infty$ and $M$ is, in addition,
		complete with respect to its Kobayashi distance, then $M$ is a weak
		visibility submanifold. Further, $M$ is also a geodesic visibility 
		submanifold. 
	\end{enumerate} 
\end{proposition}

\begin{proof}
	\noindent {\bf Proof of (1):} The proof of this is very similar to that
	of \cite[Proposition~2.4]{Bracci_Nikolov_Thomas}. The only difference
	is that where the authors of \cite{Bracci_Nikolov_Thomas} dealt with 
	{\em geodesics} in domains, we deal with almost-geodesics in complex 
	submanifolds. Since, apart from this difference and the consequent trivial
	modifications (in particular, the use of Result~\ref{Res:basic_but_important}
	to provide a reverse triangle inequality for triples of points lying on an
	almost-geodesic), the proofs are almost identical, we omit the proof.	
	\medskip
	
	\noindent {\bf Proof of (2):} The proof of this is very similar to that 
	of \cite[Proposition~2.5]{Bracci_Nikolov_Thomas}. The only difference is the
	one pointed out in the proof of (1) above. Since, apart from this difference,
	the consequent trivial modifications, and the taking into account of certain
	obvious facts that are straightforward analogues of corresponding facts
	used in the proof of \cite[Proposition~2.5]{Bracci_Nikolov_Thomas}, the proofs 
	are almost identical, we omit the proof. 
\end{proof}

\begin{corollary}\label{C:BNTandkappaposit}
Let $M$ be as in the above proposition, and suppose 
that it is complete with respect to its Kobayashi distance. 
Then: $M$ is a geodesic visibility submanifold $\iff$ $M$ is a weak visibility 
submanifold 
$\iff$
$M$ is a $(1,\kappa)$-visibility submanifold for some $\kappa>0$.
\end{corollary}

\begin{proof}
The proof of this corollary follows from that of 
Proposition~\ref{P:charac_visib_Grom_prod_gen} once we note (see 
\cite[Proposition~2.5]{Bracci_Nikolov_Thomas}) that $M$ being a geodesic
visibility submanifold is equivalent to the finiteness condition on the Gromov 
product appearing in Proposition~\ref{P:charac_visib_Grom_prod_gen}
(work with real Kobayashi geodesics instead of $(1,\kappa)$-almost-geodesics in 
Part~$(1)$ of Proposition~\ref{P:charac_visib_Grom_prod_gen}).
\end{proof}

The hypotheses of the next proposition resemble a few of those in
Theorem~\ref{T:EVL}. The proposition provides a sufficient condition weaker than 
the one occurring in Proposition~\ref{P:charac_visib_Grom_prod_gen} 
for a submanifold to be a weak visibility submanifold.

\begin{proposition} \label{P:suff_cond_visib_Grom_prod_excp_set}
Let $ M $ be as above. 
Let $ E \subset \partial M $ be a closed set such that for any $ 
p \neq q \in \partial M $, there exist $ p'\in \partial M $ and $ 
r > 0 $ such that
$ p \in B(p', r) $, $ q \in \partial M \setminus 
\overline{B(p',r)}$ and $ E \cap \partial B(p', r) = \emptyset $. 
Suppose that for some (hence any) $ o \in M $, $ \koba_M(z, o) 
\to \infty $ as $ z \to \xi\in\partial M \setminus E $.
Suppose also that
\[
\forall \, p,q \in \bdy M \setminus E \text{ with } p \neq q, \;
\limsup_{(x,y) \to (p,q)} (x|y)_o < \infty.
\]
Then $M$ is a weak visibility submanifold. When $(M,\koba_{M})$ is complete, $M$ 
is a geodesic visibility submanifold.	
\end{proposition}

\begin{proof}
	The proof is similar to that of 
	\cite[Proposition~2.5]{Bracci_Nikolov_Thomas}. The ideas behind 
	the essential modifications required in the proof are the same as 
	those occurring in the proof of Theorem~\ref{T:EVL}. 
	\smallskip
	
	Assume, to get a contradiction, that
	there exists $\kappa \geq 0$ such that $M$ is not a 
	$(1,\kappa)$-visibility submanifold. Then there exist $p,q \in \bdy 
	M$, $p \neq q$, sequences $(p_n)_{n \geq 1}$ and $(q_n)_{n \geq 1}$ in $M$ 
	such that $p_n \to p$ and $q_n \to q$ as $n \to \infty$, and a
	sequence $(\gamma_n)_{n \geq 1}$ of $(1,\kappa)$-almost-geodesics, 
	$\gamma_n : [a_n,b_n] \longrightarrow M$, such that  
	$\gamma_n(a_n)=p_n$, $\gamma_n(b_n)=q_n$ for all $n \in \posint$ and 
	such that 
	\begin{equation*}
		\max_{a_n \leq t \leq b_n} \dtb{M}(\gamma_n(t)) \to 0 \,\,\, 
		\text{as} \,\,\, n \to \infty.
	\end{equation*}
	By hypothesis, there exist $p' \in \bdy M$ and $r>0$ such that
	$p \in B(p',r)$, $q \notin \overline{B(p',r)}$ and such that
	$\bdy B(p',r) \cap E = \emptyset$. 
	We now use the arguments in the proof of Theorem~\ref{T:EVL} and those needed
	to complete that of Proposition~\ref{P:charac_visib_Grom_prod_gen} (part~(2))
	sketched above to conclude that there exist 
	$\alpha_n,\beta_n$, $a_n < \alpha_n < \beta_n < b_n$, a point $\xi 
	\in \bdy B(p',r) \cap \bdy M$, a neighbourhood $U$ of $\xi$, and
	$t_n \in [\alpha_n,\beta_n]$ such that $\overline{U} \cap E = 
	\emptyset$ and such that, writing $\sigma_n \defeq 
	\gamma_n|_{[\alpha_n, \beta_n]}$, $\xi_n \defeq 
	\sigma_n(\alpha_n)$, $\eta_n \defeq \sigma_n(\beta_n)$ and $w_n 
	\defeq \sigma_n(t_n)$, we have
	\begin{itemize}[leftmargin=14pt]
	\item $\sigma_n([\alpha_n,\beta_n]) \subset U$, $\xi_n \to \xi$ as 
	$n\to\infty$ and $\eta_n$ converges to some point $\eta\neq\xi$ of 
	$U \cap \bdy M$;
	\smallskip
	
	\item for all $n\in\posint$, $\|\xi_n-w_n\|=\|\eta_n-w_n\|$ and $(w_n)_{n \geq 1}$ converges to some point $w$ of $U \cap 
	\bdy M$ that satisfies $\|w-\xi\|=\|w-\eta\|$.
	\end{itemize}
	
	Since $\xi,w$ and $\eta$ are all distinct points of $\bdy M \setminus E$ and 
	$(\xi_n)_{n \geq 1}$, $(w_n)_{n \geq 1}$ and $(\eta_n)_{n \geq 1}$ converge to 
	$\xi,w$ and $\eta$, respectively, therefore, by 
	hypothesis, there exists $C < \infty$ such that
	\begin{equation*}
		\limsup_{n \to \infty} \ (\xi_n|w_n)_o \leq C \text{ and } \limsup_{n 
			\to \infty} \ (w_n|\eta_n)_o \leq C.
	\end{equation*}
	Therefore we may, without loss of generality, suppose that there exists
	$C<\infty$ such that, for all $n \in \posint$, $2(\xi_n|w_n)_o \leq C$ 
	and $2(w_n|\eta_n)_o \leq C$.
	From this point on, we argue exactly as in the concluding part of the proof of
	\cite[Proposition~2.5]{Bracci_Nikolov_Thomas} (replacing geodesics by 
	almost-geodesics and using Result~\ref{Res:basic_but_important} to obtain
	reverse triangle inequalities where needed) to obtain the contradiction that
	$\limsup_{n \to \infty}\koba_M(w_n,o)<\infty$ (recall that $(w_n)_{n \geq 1}$ 
	converges to $w \in \bdy M \setminus E$). This contradiction shows that $M$
	must be a weak visibility submanifold. It is also clear (we simply work
	with geodesics instead of $(1,\kappa)$-almost-geodesics) that, 
	when $(M,\koba_{M})$ is complete, it is a geodesic visibility submanifold.       		
\end{proof}

\begin{remark} \label{rm:act_visib_quasi-geod}
The proof above actually shows that under the hypotheses of the above
proposition, $M$ satisfies the visibility property with respect to
continuous $(1,\kappa)$-quasi-geodesics.
\end{remark}

Before we state the next corollary, we need two definitions. The first
generalizes Definition~\ref{Def:C-str-cvx_smoo} to the case of convex
domains whose boundaries are not necessarily smooth (it is not 
difficult to check that the following definition is consistent with
Definition~\ref{Def:C-str-cvx_smoo}).

\begin{definition} \label{def:C-str-cvx_gen}
Given a convex domain $\OM \subset \C^d$, a boundary point $p$ of 
$\OM$ is said to be {\em a $\C$-strictly convex boundary point} if for
every complex affine line $L$ such that $L \cap \OM = \emptyset$ and
such that $p \in L$, $(L \cap \bdy \OM) \setminus \{p\} = \emptyset$.
\end{definition}

We now give the following definition, which we have adopted from 
\cite{Bracci_Nikolov_Thomas} (see  
\cite[Definition~6.12]{Bracci_Nikolov_Thomas}). 

\begin{definition} \label{def:loc_C-str-cvx}
Given a domain $\OM \subset \C^d$, a boundary point $p$ of $\OM$ is
said to be {\em locally $\C$-strictly convex} if there exists a
bounded $\C^d$-neighbourhood $U$ of $p$ and a biholomorphism $\Psi:
U \to \Psi(U)$ such that $\Psi(U \cap \OM)$ is a convex domain and
such that $\Psi(p)$ is a $\C$-strictly convex boundary point of
$\Psi(U \cap \OM)$.  
\end{definition}

We are now ready to state and prove the following corollary.
\begin{corollary}\label{cor:C-str-cvx_Dini_imp_visib}
	Let $\OM$ be a bounded domain in $ \C^d$. Let $E \subset \bdy\OM$ 
	be a closed set such that
	for any $p, q \in \bdy\OM$, $p \neq q$, there exist $p'\in \bdy\OM$ 
	and $r > 0$ such that $p \in B(p',r)$,
	$q \in \bdy\OM \setminus \overline{B(p', r)}$ and $E \cap 
	\bdy B(p', r) = \emptyset$. Further, assume that 
	every $q' \in \bdy\OM \setminus E$ is both locally $\C$-strictly 
	convex and a $\smoo^{1,\,{\rm Dini}}$-smooth boundary point. 
	Then $\OM$ is a weak visibility domain. 
	Further, if $(\OM,\koba_\OM)$ is complete, $\OM$ is a geodesic
	visibility domain.  
\end{corollary}

\begin{proof}
	By Proposition~\ref{P:suff_cond_visib_Grom_prod_excp_set},
	if we can show that
	$\koba_\OM(o,z) \to \infty$ as $z$ tends to an arbitrary point of 
	$\bdy\OM \setminus E$ and that, for every $p,q \in \bdy\OM \setminus E$
	with $p\neq q$, $\limsup_{(x,y) \to (p,q)} (x|y)_o < \infty$, the
	result will be proved. 
	\smallskip
	
	Firstly note that, since every $p \in 
	\bdy \OM \setminus E$ is locally $\C$-strictly convex, every such
	point is also, by \cite[Theorem~6.13]{Bracci_Nikolov_Thomas}, and
	to use the terminology of 	
	\cite[Definition~6.1]{Bracci_Nikolov_Thomas},  
	a $k$-point. This means that
	\begin{itemize}
		\item[(*)] $\forall \,$ neighbourhood $W$ of $p$, 
		$\liminf_{z \to p} \big( \koba_\OM(z,\OM \setminus W) - (1/2)\log\big(1/\dtb{\OM}(z)\big) > -\infty.$ 
	\end{itemize}
	In particular, $\lim_{z \to p} \koba_\OM(o,z)=\infty$.
	But (*) also implies (see \cite[Theorem~6.13]{Bracci_Nikolov_Thomas}) 
	that, if $p$ and $q$ are a pair of distinct points
	in $\bdy \OM \setminus E$, then they satisfy the log-estimate (see 
	\cite[equation~2.5]{Bracci_Nikolov_Thomas}), i.e., there exist 
	neighbourhoods $V$ and $W$ of $p$ and $q$, respectively, in $\C^d$, and
	$C < \infty$ such that, for every $x \in V \cap \OM$ and every 
	$y \in W \cap \OM$,
	\begin{equation} \label{eq:kob_xy_stan_est}
		\koba_\OM(x,y) \geq (1/2) \log\big(1/\dtb{\OM}(x)\big) + (1/2) \log \big( 1/\dtb{\OM}(y) \big) - C.
	\end{equation}
	Further, it is a straightforward consequence of 
	\cite[Theorem~7]{Nikolov_Lyubomir} that we may choose $V$ and $W$ to be
	so small that there exists $C_1 < \infty$ such that, for every $x \in
	V \cap \OM$ and every $y \in W \cap \OM$,
	\begin{align*}
		\koba_\OM(o,x) &\leq (1/2)\log\big(1/\dtb{\OM}(x)\big) + C_1 \text{ and}
		\\
		\koba_\OM(o,y) &\leq (1/2)\log\big(1/\dtb{\OM}(y)\big) + C_1.
	\end{align*}
	Adding the two inequalities above, we obtain:
	\begin{equation*}
		\koba_\OM(o,x)+\koba_\OM(o,y) \leq (1/2)\log\big(1/\dtb{\OM}(x)\big) +
		(1/2)\log\big(1/\dtb{\OM}(y)\big) + 2C_1.
	\end{equation*}
	Combining the inequality above with \eqref{eq:kob_xy_stan_est}, we get
	$2(x|y)_o\leq 2C_1+C$.
	This shows that \linebreak $\limsup_{(x,y) \to (p,q)} (x|y)_o < \infty$. 
	By this, the fact that $\lim_{z \to p} \koba_\OM(o,z)=\infty$, 
	and the remark made at the beginning of the proof, the proof of the
	corollary is complete.  
\end{proof}
\medskip

\subsection{Comparison between visibility and geodesic visibility}\label{SS:comp_BZ_BNT}

Let $M$ be a bounded, connected, embedded complex submanifold of $\C^d$ such that 
$(M,\koba_M)$ is complete. Suppose that $M$ possesses the visibility property. 
Then, in particular, it possesses the weak visibility property. 
Corollary~\ref{C:BNTandkappaposit} then implies that $M$ possesses the geodesic
visibility property.
\smallskip

The following proposition shows that in the presence of Gromov hyperbolicity of 
$(M,\koba_M)$, visibility and geodesic visibility are equivalent.

\begin{proposition}\label{P:BNT&lambdakappavisibility}
Suppose that $M$ is a bounded, connected, embedded complex submanifold of $\C^d$ 
such that $(M,\koba_M)$ is a complete Gromov hyperbolic distance space. Then $M$ 
is a visibility submanifold if and only if it is a geodesic visibility submanifold.
\end{proposition}

\begin{proof}
That visibility implies geodesic visibility (in the presence of 
completeness) is clear as argued 
above. Note that we do not need Gromov hyperbolicity for this 
implication.
\smallskip

Conversely, in case $(M,\koba_{M})$ is complete and Gromov hyperbolic, the 
Geodesic Stability Theorem \cite[Chapter~III.H, Theorem~1.7]{Bridson_Haefliger}
(which states, roughly speaking, that in Gromov hyperbolic spaces geodesics and
quasi-geodesics are Hausdorff-uniformly close)
implies easily that if $M$ satisfies the visibility property with respect to
geodesics, then it also satisfies the visibility property with respect to all
$(\lambda,\kappa)$-quasi-geodesics, and hence, in particular, that it is a
visibility submanifold.
\end{proof}
\begin{remark}\label{Rm:nongromov_stillvisible}
Bharali--Zimmer constructed convex Goldilocks domains that are not 
Gromov 
hyperbolic (see, for example, \cite[Lemma~2.9]{Bharali_Zimmer}). In 
the next section, we construct two examples of 
convex domains that are not Goldilocks but that possess versions of 
the visibility property. Namely, the first example possesses the 
geodesic visibility property, whereas the second one possesses the 
(full-fledged) visibility property. 
\end{remark}

\section{Two examples}\label{S:twoexam}
In this section, we present two examples of bounded convex domains that are not 
Goldilocks domains; more precisely, condition~(1) in 
Definition~\ref{D:Goldilocks_domain} is not satisfied for either domain.
The domain in the first example is a weak visibility domain, while the 
domain in the second example is a visibility domain. We emphasize that it does 
{\bf not} seem to be easy to either prove or disprove that the domain in the
first example satisfies the visibility property.

\subsection{Example of a weak visibility domain that does not satisfy condition~(1) in Definition~\ref{D:Goldilocks_domain}.}\label{SS:BNTbutnotgoldilocks}
Consider $\Phi_0:\C^2\longrightarrow\mathbb{R}$ defined by
\[
\Phi_0(z)\defeq
\begin{cases}
\exp\big(\!\!-{1}/{|z_1|^2}\big)-\iprt(z_2), & \text{if } z_1\neq 0,\\
-\iprt(z_2), & \text{if } z_1= 0.
\end{cases}
\]
There exists an $\epsilon>0$ such that $\Phi_0$ is convex in the ball $B(0, 2\epsilon)$ (in fact, any $\epsilon<1/\sqrt{6}$ will work).
We now choose a $\smoo^{\infty}$ function $\psi: \C^2 \lraw [0,1]$ 
such that $\psi \equiv 1$ on $B(0,2\epsilon)$ and such that  
$\mathsf{supp} \, \psi \subset B(0,3\epsilon)$.
We let $\Phi \defeq \Phi_0 \cdot \psi$ 
and we also let $c_0\defeq\sup_{z\in\C^2}\big(\!\!-\Phi(z)\big)=\sup_{B(0, 3\epsilon)}\big(\!\!-\Phi(z)\big)>0.$ 
\smallskip

For $n\geq 3$, we 
consider $\chi:\R\longrightarrow\R$ defined by 
$\chi(t)=(t-\epsilon^2)^n$ for all $t>\epsilon^2$ and $0$ otherwise.
Let $c_1 \defeq \inf_{t \geq (3\epsilon/2)^2} \chi(t)  $, and set $ C \defeq c_0/c_1 $. Define 
\[
\Psi(z) \defeq C\,\chi(\Vert z \Vert^2) \ \ \forall z \in \C^2;
\]
and observe:
\begin{itemize}[leftmargin=14pt]
	\item $\Psi $ is a $\smoo^2$-smooth non-negative, convex function on $ \C^2 $ that is 
	equal to zero on $ \overline{B(0, \epsilon)}$, strongly convex locally and strictly positive on $ \C^2 \setminus \overline{B(0, \epsilon)}$.
	\smallskip
	
	\item $ \Psi(z) \geq c_0  $ for all $ z \in \C^2 \setminus\big(B(0, 3\epsilon/2))$. Hence $ \Psi(z) + \Phi(z)\geq 0$
	 $\forall z \in \C^2 \setminus\big(B(0, 3\epsilon/2)\big)$.
	 \smallskip
	 
	\item For any $z \in B(0, \epsilon)$,  $\Psi(z) + \Phi(z) = \Phi(z)=\Phi_0(z)$.
\end{itemize}

Now consider the domain 
\[
\Omega\defeq\big\{(z_1, z_2)\in\C^2 : \rho(z)\defeq\Psi(z)+\Phi(z) <0\big\}.
\]
Note that $\Omega\subset B(0, 3\epsilon/2)$, where $\rho=\Psi+\Phi_0$
is convex; consequently, $\Omega$ is a bounded convex domain.
By computing the gradient of $\rho$, we see that there exists at most 
one point $p_0\in \bdy\OM$ where the gradient vanishes, and this point
is of the form $p_0=(0,ic)$.
Moreover, $p_0\in\overline{B(0, 3\epsilon/2)}\setminus{\overline{B(0, 
\epsilon)}}$. It follows then that $\OM$ is a bounded convex domain such
that $\bdy\OM\setminus\{p_0\}$ is at least $\mathcal{C}^2$-smooth. It 
is also clear that any point 
$x\in\big(\bdy\OM\setminus\{p_0\}\big)\cap\big(
\overline{B(0, 3\epsilon)}\setminus{\overline{B(0, \epsilon)}}\big)$ is 
a strongly convex boundary point of $\OM$.
\medskip

\noindent{{\bf $\OM$ does not satisfy condition~(1) in Definition~\ref{D:Goldilocks_domain}.}}
We show formally that $\Omega$ does not satisfy condition~(1) in the 
definition of a Goldilocks domain. So what we need to do is show that
for every $\epsilon_0>0$ sufficiently small, $\int_{0}^{\epsilon_0} (\kobMetGrM_\OM(r)/r) 
dr = \infty$. 
The way our domain $\OM$ has been defined, there exists $r_0>0$ such 
that
\begin{equation*}
	\OM \cap B(0,r_0) = \big\{(z_1,z_2) \in B(0,r_0) : \iprt(z_2)>
	\exp(-1/|z_1|^2)\big\}.
\end{equation*}
Fix $\epsilon_0 \in (0,r_0)$. It is immediate that for every $r\in (0,\epsilon_0)$,
$\dtb{\OM}\big((0,ir)\big) \leq r$ (consider the boundary point $0_{\C^2}$ of 
$\OM$). 
Now we use the elementary upper bound on the Kobayashi metric to write
\begin{equation*}
	\dkoba_\OM\big((0,ir);(1,0)\big) \leq 1/\Delta_\OM\big((0,ir);(1,0)\big),
\end{equation*}
where $\Delta_\Omega(z; v)\defeq\sup\big\{t>0\mid\big(z+(t\unitdisk)({v}/{\|v\|})\big)\subset\OM\big\}$.
From the explicit description of $\OM \cap B(0,r_0)$, it follows that 
$\Delta_\OM\big((0,ir);(1,0)\big)=1/\sqrt{\log(1/r)}$. Therefore
\begin{equation*}
	\kobMetGrM_\OM(r) \geq \frac{1}{\dkoba_\OM\big((0,ir);(1,0)\big)} \geq 
	\Delta_\OM\big((0,ir);(1,0)\big) = 1/\sqrt{\log(1/r)}.
\end{equation*}
Since
\begin{equation*}
	\int_{0}^{\epsilon_0} \frac{dr}{r\sqrt{\log(1/r)}} = \infty,
\end{equation*}
$\OM$ is \emph{not} a Goldilocks domain.
\medskip

\noindent{\bf Every point of $\bdy\OM$ except possibly $p_0$ is $ \C $-strictly 
convex.}
Consider
\begin{equation} \label{dfn:S}
S\defeq\bdy\OM\cap\overline{B(0, \epsilon)}=\overline{B(0,\epsilon)} 
\cap \big\{z\in\C^2 : \Phi_0(z)=0\big\}.
\end{equation}

Then, as noticed earlier, any $p\in \partial\OM \setminus S$, $p\neq 
p_0$, is a strongly convex, and therefore also a $\C$-strictly 
convex boundary point.
Next, we shall show that any $p\in S$ is also a $\C$-strictly 
convex boundary point. This will establish that every point of 
$\bdy\OM$ except possibly $p_0$ is a $\C$-strictly convex boundary
point.
An easy computation, taking into account the fact that $\Psi \equiv 0$ on 
$B(0,\epsilon)$, shows that, for all $p\in S$,
\begin{equation}
H_p(\bdy\OM)  
= \Big\{ \xi = (\xi_1, \xi_2) \in \C^2: s(p_1)\bar p_1\xi_1 - (1/2i)\xi_2 = 0 
\Big\} = \mathsf{span}_{\C}\big\{(1,\, 2is(p_1)\bar{p}_1)\big\},
\end{equation}
where $s(p_1) \defeq \exp(-1/\vert p_1\vert^2)/\vert p_1 \vert^4$ 
for all $p_1 \neq 0$, and where $s(0) \defeq 0$. (In particular, when $p_1=0$, $H_p(\bdy\OM)=\mathsf{span}_{\C}\big\{(1,0)\big\} = 
\C\times\{0\}$.)
Write $u(p) \defeq 2is(p_1)\bar{p}_1$ for all $p \in S$. Then, for every $p\in S$, 
$H_p(\bdy\OM) = \mathsf{span}_{\C} \big\{ (1,u(p)) \big\}$. From this it follows 
that $\OM$ fails to be $\C$-strictly convex at some point of $S$ if and only if 
there exist $p\in S$ and $\zeta \in \C \setminus \{0\}$ such that $p + \zeta 
(1,u(p)) \in \bdy\OM$. From this it follows easily that
\[ \forall\,t\in [0,1],\; p + t \zeta (1,u(p)) \in S. \]
Since $S$ is as given in \eqref{dfn:S}, this implies that
\[ \forall\,t\in [0,1],\; \Phi_0\big(p + t \zeta (1,u(p))\big) = 0. \]
But, writing down the definition of $\Phi_0$ and recalling that $\zeta\neq 0$, we
see that this yields an immediate contradiction. From this contradiction it 
follows that $\OM$ is $\C$-strictly convex at every point of $S$, hence (recalling 
what we observed previously) that $\OM$ is $\C$-strictly convex at every boundary
point except, possibly, $p_0$.
\smallskip

Now, since $\OM$ also has at least $\smoo^2$-smooth boundary, we see that we may 
appeal to Corollary~\ref{cor:C-str-cvx_Dini_imp_visib} to conclude that $\OM$ is a 
weak
visibility domain. We emphasize that it is unclear whether $\OM$ is a
visibility domain. The reason is that all boundary points 
of the type
$(0,t)$ with $t$ real, $|t|<\epsilon$, are points of infinite type, as
can easily be checked. Therefore we cannot invoke any known theorem to 
conclude visibility.
	
\subsection{Example of a domain that satisfies the condition in Corollary~\ref{C:finitetypeexceptsmallset} but does not satisfy condition~(1) in Definition~\ref{D:Goldilocks_domain}} \label{subsec:BZ_VisibNotGold}
Consider $\Phi_0:\C^2\longrightarrow\mathbb{R}$ defined by
\[
\Phi_0(z)\defeq
\begin{cases}
\exp\big(\!\!-{1}/{\|z\|^2}\big)-\iprt(z_2), & z\neq 0,\\
0, &  z= 0,
\end{cases}
\]
where $\|z\|$ denotes the Euclidean norm of $z \in \C^2$.
There exists an $\epsilon>0$ such that $\Phi_0$ is convex in the ball 
$B(0, 2\epsilon)$ (in fact, any $\epsilon<1/2\sqrt{2}$ will work).
We now choose a $\smoo^{\infty}$ function $\psi: \C^2 \lraw [0,1]$ 
such that $\psi \equiv 1$ on $B(0,2\epsilon)$ and such that  
$\mathsf{supp} \, \psi \subset B(0,3\epsilon)$.
We let $\Phi \defeq \Phi_0 \cdot \psi$ and we also let
$c_0\defeq\sup_{z\in\C^2}\big(\!\!-\Phi(z)\big)=\sup_{B(0,3\epsilon)}
\big(\!\!-\Phi(z)\big)>0$. 
\smallskip

We choose another function $\chi:[0,\infty) \lraw [0,\infty)$ that is
(1) $\smoo^{\infty}$, (2) identically $0$ on $[0,\epsilon^2]$ and (3)
strictly increasing on $[\epsilon^2,\infty)$ and strongly convex on 
$(\epsilon^2,(\epsilon+\delta)^2)$ (that is, has double derivative 
positive) for some small $\delta>0$ (for example, one could take 
$\chi=\exp\big(\!-1/(t-\epsilon^2)\big)$ when $t> \epsilon^2$ and $0$ 
otherwise).
Let $c_1 \defeq \inf_{t \geq (\epsilon+\delta/2)^2} \chi(t)$, and set 
$C \defeq c_0/c_1$. Define 
\[
\Psi(z) \defeq C \, \chi(\Vert z \Vert^2) \ \ \forall \, z \in \C^2;
\]
and observe:
\begin{itemize}[leftmargin=14pt]
\item $\Psi$ is a $\smoo^{\infty}$-smooth, non-negative
function on $\C^2$ that is equal to zero on $\overline{B(0,\epsilon)}$ 
and that is strongly convex and strictly positive on $B(0, \epsilon+\delta) \setminus 
\overline{B(0,\epsilon)}$.
\smallskip
		
\item $\Psi(z) \geq c_0$ for all $z \in \C^2 \setminus 
B(0,\epsilon+\delta/2)$. Hence $\Psi(z) + \Phi(z)\geq 0$
$\forall z \in \C^2 \setminus B(0, \epsilon+\delta/2)$.
\smallskip
		
\item For any $z \in B(0, \epsilon)$,  $\Psi(z) + \Phi(z) = 
\Phi(z)=\Phi_0(z)$.
\end{itemize}
	
Now consider the domain 
\[
\Omega\defeq\big\{z=(z_1, z_2)\in\C^2 : \rho(z)\defeq\Psi(z)+\Phi(z) 
<0\big\}.
\]
Note that $\Omega\subset B(0, \epsilon+\delta/2)$, on which 
$\rho=\Psi+\Phi_0$ is convex; consequently, $\Omega$ is a bounded convex domain.
By computing the gradient of $\rho$, we see that there 
exists at most one point $p_0\in \bdy\OM$ where the gradient 
vanishes, and this point is of the form $p_0=(0,ic)$.
Moreover, $p_0\in\overline{B(0,\epsilon+\delta/2)} \setminus 
{\overline{B(0,\epsilon)}}$. It follows then that $\OM$ is a bounded 
convex domain such that $\bdy\OM\setminus\{p_0\}$ is 
$\smoo^{\infty}$-smooth. It is also clear that any point 
$x \in \big( \bdy\OM \setminus \{p_0\} \big) \cap \big(
\overline{B(0,\epsilon+\delta/2)} \setminus \overline{B(0, \epsilon)} \big)$ is 
a strongly convex boundary point of $\OM$, whence it is of finite 
type.
Set $S \defeq \bdy \OM \cap \overline{B(0,\epsilon)}$ and observe that
\begin{equation} \label{eq:form_of_S}
S = \overline{B(0,\epsilon)} \cap \big\{ z\in\C^2 : \Phi_0(z)=0 \big\}.	
\end{equation}
It is also easy to show that any point in $S$ different from $0$ is a boundary 
point of $\OM$ of finite type. Hence, using 
Corollary~\ref{C:finitetypeexceptsmallset}, it follows that $\OM$ is a visibility 
domain. That $\OM$ is a geodesic visibility domain follows easily from 
Corollary~\ref{C:BNTandkappaposit}
\medskip
	
\noindent{\bf $\Omega$ does not satisfy condition~(1) in Definition~\ref{D:Goldilocks_domain}.}
Note that
\begin{equation*}
\OM \cap B(0,\epsilon/2) = \big\{(z_1,z_2) \in B(0,\epsilon/2) : 
\iprt(z_2) > \exp\big(-1/\|z\|^2\big)\big\}.
\end{equation*}
From the above expression,
we see that, for $r$ sufficiently small, $p_r \defeq (0,ir) \in \OM$. 
Write $v \defeq (1,0)$; we regard $v$ as a unit vector in $\C^2$. It is 
easy to see that $\Delta_{\OM}(p_r,v) \geq \rho$, where $\rho$ is given by
$\rho = \sqrt{\big(1/\log(1/r)\big)-r^2}$.
Therefore, using arguments similar to those used in dealing with 
Example~\ref{SS:BNTbutnotgoldilocks}, we readily obtain $\kobMetGrM_\OM(r) \geq 
1/\dkoba_\OM(p_r,v) \geq \rho = \sqrt{\big(1/\log(1/r)\big)-r^2}$.
Therefore, to prove that $\OM$ does not satisfy Condition~1 in the
definition of a Goldilocks domain, it suffices to show that for $\delta
> 0$ sufficiently small so that the integrand makes sense,
\begin{equation*}
\int_0^\delta \frac{1}{r} \sqrt{\frac{1}{\log(1/r)}-r^2}\, dr = \infty. 
\end{equation*}
This follows easily.
Therefore $\OM$ is not a Goldilocks domain.
\medskip
  
\section{Properties of Visibility Subspaces and the Continuous Extension of Kobayashi Isometries} \label{S:Ext_Kob_Iso}

In this section we shall make the requisite comments about the proof of 
Theorem~\ref{th:ext_of_id_vis}, and also prove
Theorem~\ref{th:cont_ext_koba_isom} and related corollaries. In the first 
subsection below, we make certain observations regarding geodesic subspaces and 
also present two lemmas about visibility subspaces, which will be needed in 
the proofs of the aforementioned theorems. In the next subsection, we deal with
the proofs proper.
\smallskip

As promised in the Introduction, we first provide a sketch of an argument 
showing why every subspace $V_f$ of
$\unitdisk^n$, $n\geq 2$, of the form $V_f = \{(z,f(z)) : z \in \unitdisk\}$,
where $f=(f_1,\dots,f_{n-1}) : \unitdisk \lraw\unitdisk^{n-1}$ is a
holomorphic map, is a visibility subspace if $f$ extends continuously
to $\overline{\unitdisk}$. 
It is easy to see that every
point of $\bdy_a V_f = \overline{V}_f \setminus V_f$ is of the form 
$(\zeta,f(\zeta))$, where $|\zeta|=1$. Using this fact, the visibility property 
of $\unitdisk$, the explicit form of $\koba_{\unitdisk^n}$, and the
distance-decreasing property of holomorphic maps with respect to the Kobayashi
distance, it is now easy to show that any pair of distinct points of $\bdy_a 
V_f$ satisfies the visibility property with respect to geodesics of 
$\koba_{{\unitdisk}^n}$.

\subsection{Preliminary observations regarding geodesic subspaces and two preparatory lemmas}
\label{SS:prelim_obs_geo_sub}  
Given a geodesic subspace $S$ of $(M,\koba_M)$, it is easy to see from the
definition that $S$ is closed in $M$, that $(S,\koba_{M}|_{S\times S})$
is locally compact, and that the 
other crucial hypothesis in \cite[Theorem~8.1]{Kobayashi} is satisfied. 
Therefore, by this latter result, the completeness of $(S,\koba_M|_{S 
\times S})$ is equivalent to the condition that every closed ball 
is compact, i.e., the distance space is proper. In particular, if
$z_0$ is any fixed point of $S$, if $p$ is a fixed but arbitrary
point of $\bdy_a S \defeq \overline{S} \setminus S$, and if
$(z_n)_{n \geq 1}$ is a sequence in $S$ converging, in the 
Euclidean sense, to $p$, then $\lim_{n \to \infty} 
\koba_M(z_0,z_n)=\infty$. This latter fact also implies that
$\bdy_a S\subset\bdy M=\overline{M}\setminus M$.
\smallskip

Let $M=D$, a bounded domain in $\C^d$.
If $(D,\koba_D)$ is a complete distance space, then any 
closed subset $S$ of $D$ satisfies the first defining condition of a geodesic
subspace. 
Thus, in this case, every closed subset $S$ of $D$ that satisfies 
only the second condition in Definition~\ref{Def:geodspace} will be a geodesic subspace. For example:
every holomorphic retract of a complete distance space 
$(D,\koba_D)$ is a geodesic subspace.
Let $\OM$ and $D$ be bounded domains in $\C^m$ and
$\C^n$, respectively, such that $(\OM,\koba_\OM)$ is complete.
Let $f: \OM \lraw D$ be an isometry with
respect to the Kobayashi distances (we are not making any claims
about the existence of such isometries). If we write $S \defeq
f(\OM) \subset D$, then it is easy to see that $S$ is a geodesic 
subspace of $D$. This example suggests that there could 
be a bounded domain $D$ such that $(D,\koba_D)$ is not complete, but
such that $D$ nevertheless has geodesic subspaces. Indeed, this is 
the case. 
\begin{example}
Let $\OM \subset \C^d$ be a bounded convex domain and let $A$ be an
analytic subset of $\OM$ of co-dimension at least 2. Let $D \defeq \OM
\setminus A$. Then $\koba_D = \koba_{\OM}|_{D \times D}$ 
(see, for example, \cite[Theorem~2]{Campbell_Ogawa}). Choose a complex geodesic 
$f$ in $\OM$ that avoids $A$. Clearly, $f$ is a complex geodesic in 
$D$ too. Note that $(D, \koba_D)$ is not complete. But $f(\unitdisk)$ is a 
geodesic subspace of $D$.
\end{example}

We need a definition. Before providing it, we clarify that, if $X$ is
a given topological space, then by a compactification of $X$ we shall
mean a pair $(\iota,\widetilde{X})$, where $\widetilde{X}$ is 
required to be a compact Hausdorff topological space and $\iota : X 
\lraw \widetilde{X}$ is required to be a homeomorphism onto its image
$\iota(X)$, which is, in addition, required to be an open, dense
subset of $\widetilde{X}$. We shall regard $X$ as being a subset of
$\widetilde{X}$ (by identifying $X$ with $\iota(X)$). 
\begin{definition} \label{dfn:geod_loop}
Let $(X,d)$ be a proper, geodesic distance space and let 
$(\iota,\widetilde{X})$ be a
compactification of $X$ (regarded as a topological space with the
topology induced by $d$). By a {\em geodesic loop of} $X$ {\em in} 
$\widetilde{X}$ we mean a geodesic line $\gamma$ in $(X,d)$ (that 
is, an isometric embedding $\gamma$ from $(\R,|\cdot|)$ to $(X,d)$) such 
that the set of limit points of $\gamma$ at $\infty$ is equal to the 
set of limit points of $\gamma$ at $-\infty$. (Note that the set 
of limit points of $\gamma$ at $\infty$ (and $-\infty$) is contained 
in $\widetilde{X} \setminus X$.) 	
\end{definition}
\noindent We point out that we will only use this notion in the case
where $X=S$ is a geodesic subspace of a bounded, connected, embedded
submanifold $M$ of $\C^d$.
\smallskip

We note that it is easy to define the notion of visibility for a pair consisting 
of a proper geodesic distance space $(X,d)$ and a compactification 
$(\iota,\widetilde{X})$ of $X$, by analogy with Definition~\ref{Def:vissub}. The 
important thing for us to note is that if $X$ has the visibility property with 
respect to the compactification $\widetilde{X}$, the proof of the first part of
\cite[Lemma~3.1]{Bracci_Nikolov_Thomas}
goes through without change to show that every 
geodesic ray $\gamma$ in $(X,d)$ (i.e., an isometric embedding $\gamma:
\big([0,\infty),|\cdot|\big) \lraw (X,d)$) {\em lands} at a point of
$\widetilde{X} \setminus X$, i.e., $\lim_{t \to \infty} \gamma(t)$
exists as an element of $\widetilde{X} \setminus X$ (which is the boundary of $X$ 
in $\widetilde{X}$). We note that, in such a situation, a geodesic loop of $X$ in 
$\widetilde{X}$ is a geodesic line $\gamma$ such that 
$\lim_{t \to -\infty} \gamma(t) = \lim_{t \to \infty} \gamma(t)$.
\smallskip

We now state two lemmas, the second of which is a mild generalization of 
\cite[Lemma~3.1]{Bracci_Nikolov_Thomas} and which was referred to above. The 
utility of these lemmas will become apparent when we prove 
Theorem~\ref{th:ext_of_id_vis}. Since the proof of the second lemma is 
substantially the
same as that of \cite[Lemma~3.1]{Bracci_Nikolov_Thomas}, we omit the proof.  
The essential observation here is that the proof in \cite{Bracci_Nikolov_Thomas} 
goes through virtually without modification in the more general setting of 
visibility subspaces. 
	
\begin{lemma} \label{L:vis_dst_bdy_pts_infty}
Suppose that $M \subset \C^d$ is a bounded, connected, embedded
complex 
submanifold of $\C^d$ and that $S$ is a visibility subspace of $M$. 
If $(z_\nu)_{\nu \geqslant 1}$ and $(w_\nu)_{\nu \geqslant 1}$ are 
sequences in $S$
converging to \emph{distinct} boundary points $p,q \in \bdy_a S$, 
then $\koba_M(z_\nu,w_\nu) \to \infty$ as $\nu \to \infty$.
\end{lemma}
\begin{proof} 
Note that the proof of (1) of Proposition~\ref{P:charac_visib_Grom_prod_gen} goes 
through with almost no modifications to show that 
$\limsup_{\nu\to\infty}(z_\nu|w_\nu)_o<\infty$,
where the Gromov product is now calculated with respect to $\koba_M|_{S\times S}$ 
and for any fixed $o\in S$. 
Combining this with our previous observation that $\koba_M(x, o)\to\infty$ when 
$S\ni x\to x_0\in\bdy_a S$, the required conclusion follows immediately.
\end{proof}

\begin{lemma} \label{lm:conv_seq_geod}
Let $M$ and $S$ be as above.
Then any geodesic ray $\gamma$ in $S$ lands at a point $p$ of $\bdy_a S$, 
i.e., there exists $p \in \bdy_a S$ such that $\lim_{t \to \infty} 
\gamma(t) = p$.
	
Conversely, suppose that $z_0 \in S$ and that 
$(z_{\nu})_{\nu \geqslant 1}$ is a sequence in $S$ converging to a
point $p \in \bdy_a S$. For every $\nu$, let $\gamma_{\nu}$ be a
$\koba_M$-geodesic in $S$ joining $z_0$ to $z_{\nu}$. Then, up to a
subsequence, $(\gamma_{\nu})_{\nu \geqslant 1}$ converges uniformly
on the compact subsets of $[0, \infty)$ to a geodesic ray that lands 
at $p$. 
\end{lemma}

\subsection{The proofs of Theorem~\ref{th:ext_of_id_vis} and Theorem~\ref{th:cont_ext_koba_isom}}\label{SS:ext_of_id_vis}

With these two lemmas in place, it is easy to see that one can, without any 
difficulty, replicate the arguments in the proof of 
\cite[Theorem~3.3]{Bracci_Nikolov_Thomas} to prove 
Theorem~\ref{th:ext_of_id_vis}.
We therefore omit the proof of the latter.
\smallskip

We now illustrate the usefulness of Theorem~\ref{th:ext_of_id_vis} by proving 
Theorem~\ref{th:cont_ext_koba_isom}.

\begin{proof}[The proof of Theorem~\ref{th:cont_ext_koba_isom}]
Note that $f$ is an isometry between $(M,\koba_M)$ and
$(S,\koba_N|_{S \times S})$. Therefore, $(S,\koba_N|_{S \times S})$ is
Gromov hyperbolic (because $(M,\koba_M)$ is by assumption
so). By the general theory of Gromov hyperbolic spaces
(see \cite[Part~III, Chapter~H, Theorem~3.9]{Bridson_Haefliger}), $f$ extends to  
a homeomorphism $\widetilde{f}$ from $\overline{M}^G$ to  
$\overline{S}^G$. By Theorem~\ref{th:ext_of_id_vis}, $\mathsf{id}_S : 
S \lraw S$ extends to a continuous surjection $\widehat{\mathsf{id}}_S : 
\overline{S}^G \lraw \overline{S}$. There is also a natural inclusion 
$i_{\overline{S}}$ of $\overline{S}$ in $\overline{N}$. If we define
$\widehat{f} \defeq i_{\overline{S}} \circ \widehat{\mathsf{id}}_S 
\circ \widetilde{f}$,
then it is clear that $\widehat{f}:\overline{M}^G \lraw \overline{N}$ 
is 
a continuous extension of $f$. If $S$ has no geodesic loops in
$\overline{S}$ then, again by Theorem~\ref{th:ext_of_id_vis}, 
$\widehat{\mathsf{id}}_S$ is a homeomorphism from $\overline{S}^G$ to 
$\overline{S}$ and it follows from the definition of $\widehat{f}$ 
that, regarding it as a mapping from $\overline{M}^G$ to 
$\overline{S}$, it is a homeomorphism.
\end{proof}

Now, we shall present two important corollaries of Theorem~\ref{th:cont_ext_koba_isom}. 

\begin{corollary}\label{C:LamdaKappaVisbilityExtesion}
	Suppose that $M \subset \C^m$ and $N \subset \C^n$ are bounded,
	connected, embedded complex submanifolds and that $M$ is complete 
	with respect to its Kobayashi distance. Suppose that $f:M \longrightarrow N$ 
	is an isometry with respect to the
	Kobayashi distances. Suppose that $(M,\koba_M)$ is Gromov
	hyperbolic and that $N$ is a weak visibility submanifold. 
	Then $f$ extends to a continuous map $\widehat{f} : \overline{M}^G 
	\longrightarrow \overline{N}$, where $\overline{M}^G$ denotes the 
	Gromov compactification of $(M,\koba_M)$.
\end{corollary}
\begin{proof}
	Since $N$ is a weak visibility submanifold, (1) of 
	Proposition~\ref{P:charac_visib_Grom_prod_gen} gives us: for every $p,q 
	\in \bdy N$, $p \neq q$, $\limsup_{(x,y)\to (p,q)} (x|y)_o < \infty$. In 
	particular,
	\begin{equation*}
		\forall \, p,q \in \bdy_a S \text{ with } p \neq q, \; 
		\limsup_{S \times S \ni (x,y) \to (p,q)} (x|y)_o < \infty,
	\end{equation*}
	where we take $S\defeq f(M)$, and $o$ to be an arbitrary but fixed point of 
	$S$. Now the reader can easily verify that precisely the same method that is 
	used to prove (2) of Proposition~\ref{P:charac_visib_Grom_prod_gen} can also 
	be used, in this case, to show that $S$ is a visibility subspace of $N$ (keep 
	in mind that $(S,\koba_N|_{S \times S})$ is a {\em proper} distance space). 
	So we may once again apply Theorem~\ref{th:cont_ext_koba_isom} to 
	draw the desired conclusion.  
\end{proof}
In particular, when $ M = \mathbb{D} $, which is a complete, Gromov 
hyperbolic distance space with respect to the Kobayashi distance $ 
\koba_{\mathbb{D}} $ and for which the Gromov compactification is 
known to coincide with the Euclidean compactification, we have:
\begin{corollary}\label{C:ComplexGeodesicExtension}
	Suppose that $M \subset \C^m$ is a weak visibility submanifold.
	Suppose that $f:\unitdisk \lraw M$ is a complex geodesic. Then $f$
	extends to a continuous map $\widehat{f} : \overline{\unitdisk} 
	\lraw \overline{M}$.
\end{corollary}

We are now ready to present the proof of 
Corollary~\ref{Cor:ext_cg_dinismooth_strcvx}.

\begin{proof}[The proof of Corollary~\ref{Cor:ext_cg_dinismooth_strcvx}] We first 
note that, by the arguments that occur in the first part of the proof of 
Corollary~\ref{C:finitetypeexceptsmallset}, we can show that the hypotheses 
of Corollary~\ref{cor:C-str-cvx_Dini_imp_visib} in Section~\ref{S:local_cons_EVL} 
are satisfied by $\OM$. Consequently, by 
Corollary~\ref{cor:C-str-cvx_Dini_imp_visib}, $\OM$ is a weak visibility domain. 
Now we can invoke Corollary~\ref{C:ComplexGeodesicExtension} to obtain the desired 
conclusion.
\end{proof}

\section{A Wolff--Denjoy-type theorem} \label{S:WD}
In this section we present a proof of Theorem~\ref{th:WD_submani}.
Our proof relies on two crucial ingredients that are consequences of visibility with respect to $(1,\kappa)$-almost-geodesics for
some $\kappa>0$. In the first subsection below, we present these ingredients first.
	
\subsection{Preparations}\label{SS:prepwd}	
Our first ingredient is an analogue of Proposition~4.1 in 
\cite{Bharali_Maitra}. Its proof 
is based on an argument developed by Karlsson in \cite{Karlsson}.

\begin{proposition} \label{prp:tech_lem_vis_man}
Let $M$ be a bounded, connected, embedded complex
submanifold of $\C^d$. Suppose there exists $\kappa_0> 0$ such that $M$ 
possesses the visibility property
with respect to $(1, \kappa_0)$-almost-geodesics. 
Let $\nu,\mu : \posint \lraw \posint$ be strictly increasing 
functions such that there
exists $m_0 \in M$ so that
\begin{equation} \label{eq:two_seqs_conv_infty}
\lim_{j \to \infty} \koba_M\big(F^{\nu(j)}(m_0),m_0\big) = 
\lim_{j \to \infty} \koba_M\big(F^{\mu(j)}(m_0),m_0\big) = \infty.
\end{equation}
Then there exists $\xi \in \bdy M$ such that, for all $z\in M$, $\lim_{j \to 
\infty} F^{\nu(j)}(z) = \lim_{j \to \infty} F^{\mu(j)}(z) = \xi$.
\end{proposition}
	
\begin{proof}
The proof of \cite[Proposition~4.1]{Bharali_Maitra} goes through without
modification. The only observation to be made is that the argument given there
works for weak visibility submanifolds, not just visibility domains as considered
in \cite{Bharali_Maitra}, when one takes into account the essential lemmas
regarding the Kobayashi distance and metric on submanifolds presented in 
Section~\ref{S:prelims} of this paper.
\end{proof}

Our second ingredient is Theorem~\ref{th:bdy-val_lims_cons} below, 
which is a consequence of Theorem~\ref{th:visib_taut_KobMetGr_0}.
The latter theorem says that, when $M$ is taut, visibility 
with respect to $(1,\kappa)$-almost-geodesics for some $\kappa>0$ 
implies that $\kobMetGrM_{M}(r) \to 0 \,\,\, \text{as}\,\,\, r \to 0$. 
(We recall that the notation $\kobMetGrM_{M}(r)$ is explained right after 
\eqref{E:kobMetGr}.)   
To prove Theorem~\ref{th:visib_taut_KobMetGr_0}, we need two lemmas. Both 
of them are elementary; so we state them here without proof.
\begin{lemma} \label{lm:seq_holo_uni_Lip}
	Suppose that $(\phi_{\nu})_{\nu \geq 1}$ is a sequence of
	holomorphic maps from $\unitdisk$ to $\C^d$. Suppose that $(\phi_{\nu})_{\nu \geq 1}$ is uniformly
	bounded. Then, for every $r_0 \in (0,1)$, there exists $L=L(r_0)<
	\infty$ such that
	\begin{equation*}
	\forall \, \nu \in \posint, \, \forall \zeta_1,\zeta_2 \in D(0,r_0),
	\, \|\phi_{\nu}(\zeta_1)-\phi_{\nu}(\zeta_2)\| \leq 
	L|\zeta_1-\zeta_2|,
	\end{equation*}
	where $ D(0,r_0) \defeq \{\zeta\in\C:\,|\zeta|<r_0\,\}$.
\end{lemma}
\begin{lemma} \label{lm:seq_holo_well-sep}
	Suppose that $(\phi_{\nu})_{\nu \geq 1}$ is a sequence of
	holomorphic maps from $\unitdisk$ to $\C^d$. Suppose that $(\phi_{\nu})_{\nu \geq 1}$ is uniformly
	bounded and that there exists $\epsilon_0>0$ such that, for all $\nu
	\in \posint$, $\|\phi'_{\nu}(0)\|\geq\epsilon_0$. Then there exists
	$\delta > 0$ such that
	\begin{equation*}
	\forall \, \nu \in \posint, \, \forall \, \zeta_1,\zeta_2 \in
	\clos{D(0,\delta)}, \,
	\|\phi_{\nu}(\zeta_1)-\phi_{\nu}(\zeta_2)\|\geq
	(\epsilon_0/2)|\zeta_1-\zeta_2|,
	\end{equation*} 
	where $ D(0,\delta) \defeq \{\zeta\in\C:\,|\zeta|<\delta\,\}$.
\end{lemma}
\begin{theorem} \label{th:visib_taut_KobMetGr_0}
	Let $M$ be a bounded, connected, embedded complex
	submanifold of $\C^d$. Suppose that $M$ has the visibility property with 
	respect to $(1,\kappa)$-almost-geodesics for some $\kappa>0$ and,
	moreover, that $M$ is taut. Then $\kobMetGrM_{M}(r) \to 0$ as $r \to 0$. 
\end{theorem}

\begin{proof}
	We will
	closely follow the proof of Theorem~4.2 in \cite{Bharali_Maitra}. However, we
	provide a complete proof here because there is an 
	essential difference between the proof of \cite[Theorem~4.2]{Bharali_Maitra}
	and the current proof, which is that the former {\em does not} work if it is
	only known that $M$ possesses the visibility property with respect to
	$(1,\kappa)$-almost-geodesics for {\em some} $\kappa>0$.
	\smallskip
	
	Assume,
	to get a contradiction, that $\kobMetGrM_M(r) \not \to 0$ as $r\to 0$.
	Since $\kobMetGrM_M(r)$ decreases as $r$ decreases to $0$, the above
	assumption implies that there exists $\epsilon_1>0$ such that
	$\kobMetGrM_M(r) \downarrow \epsilon_1$ as $r \downarrow 0$. Let
	$\epsilon_0 \defeq \epsilon_1/2$. Then, for every $\nu \in
	\posint$, $\kobMetGrM_M(1/\nu) > \epsilon_0$. Therefore, for every 
	$\nu \in \posint$, there exist $z_{\nu} \in M$ such that 
	$\dtb{M}(z_{\nu}) \leqslant 1/\nu$ and $v_{\nu} \in T_{z_{\nu}}^{(1,0)}M$ with
	$\|v_{\nu}\|=1$ such that $1/\dkoba_M(z_{\nu},v_{\nu})>\epsilon_0$, i.e.,
	$\dkoba_M(z_{\nu},v_{\nu}) < 1/\epsilon_0$.
    We also assume, without loss of generality, that 
    $(z_{\nu})_{\nu \geq 1}$ converges to some point $\xi \in \bdy M$.
	By the definition of $\dkoba_M$, the inequalities above imply that 
	there exist a holomorphic map $\phi_{\nu}:\unitdisk \lraw M$ such 
	that
	$\phi_{\nu}(0)=z_{\nu}$ and a $t_{\nu} \in (0,1/\epsilon_0)$ such that
	$t_{\nu}\phi'_{\nu}(0)=v_{\nu}$. This last equation implies:
	$t_{\nu}\|\phi'_{\nu}(0)\|=1$, which in turn implies that, for all $\nu\in
	\posint$, $\|\phi'_{\nu}(0)\|>\epsilon_0$. 
	By the tautness of $M$, there exists a subsequence of 
	$(\phi_{\nu})_{\nu\geq 1}$, which we continue to denote by $(\phi_{\nu})_{\nu\geq 1}$ without changing 
	subscripts, that converges uniformly on the compact
	subsets of $\unitdisk$ to a holomorphic map $\phi$ that is either 
	$M$-valued or $\bdy M$-valued. Note that
	\begin{equation*}
	\phi(0)=\lim_{\nu \to \infty}\phi_{\nu}(0) 
	=\lim_{\nu \to \infty}z_{\nu}=\xi.
	\end{equation*}
	Therefore, $\phi$ is $\bdy M$-valued. Note that
	$\|\phi'(0)\|\geq\epsilon_0$; so $\phi$ is non-constant.
	\smallskip

	Now we invoke
	Lemma~\ref{lm:seq_holo_well-sep} to conclude that there exists
	$\delta \in (0,1)$, $\delta\leq\tanh(\kappa)$, such that
	\begin{equation} \label{eq:phi-nu_well_sep}
	\forall \, \nu \in \posint, \, \forall \, \zeta_1, \zeta_2 \in
	\clos{D(0,\delta)}, \, \|\phi_{\nu}(\zeta_1)-\phi_{\nu}(\zeta_2)\| 
	\geq (\epsilon_0/2) |\zeta_1-\zeta_2|.
	\end{equation}
	Define $\eta \defeq \phi(\delta/2) \in \bdy M$ and $w_{\nu} \defeq
	\phi_{\nu}(\delta/2)$. Then $(w_{\nu})_{\nu \geq 1}$ is a sequence 
	in $M$ converging to $\eta$. By \eqref{eq:phi-nu_well_sep}, 
	it follows immediately that $\phi(0)\neq\phi(\delta/2)$, i.e., $\xi
	\neq \eta$. The sequences $(z_{\nu})_{\nu \geq 1}$ and $(w_{\nu})_{\nu \geq 1}$
	in $M$ converge to $\xi \in \partial M$ and $\eta \in \partial M$, 
	respectively.
	\smallskip
	
	Next, we shall show that $\gamma_{\nu}\defeq\phi_{\nu}\circ\sigma:[0, R] \longrightarrow M$ is a $ (1, \kappa) $-almost-geodesic in $ M $,
	where $\sigma:[0,R]\lraw\unitdisk$ is the geodesic in $\unitdisk$ for the Poincar{\'e} distance joining $0$ to $\delta/2$
	($R$ is, in fact, equal to ${\tanh}^{-1}(\delta/2)$ and $\sigma$
	itself is just $\tanh|_{[0,{\tanh}^{-1}(\delta/2)]}$). 
	By the explicit form of $\sigma$, there exists $M_\delta>1$ such that, for 
	every $s,t \in 
	[0,R]$, $(1/M_\delta)|s-t|\leq|\sigma(s)-\sigma(t)|\leq M_\delta|s-t|$.
	Now note that
    \begin{equation*}
      \forall \, \nu \in \posint, \, \forall \, s,t \in [0,R], \,
      \|\gamma_{\nu}(s)-\gamma_{\nu}(t)\| = \|\phi_{\nu}(\sigma(s))-
      \phi_{\nu}(\sigma(t))\| \leq L |\sigma(s)-\sigma(t)| \leq LM_\delta|s-t|.
    \end{equation*} 
    To write the second inequality above, we use Lemma~\ref{lm:seq_holo_uni_Lip}. 
    Therefore, all the $\gamma_{\nu}$'s are
    Lipschitz (in fact, as we see, $LM_\delta$ works as a Lipschitz constant for 
    all of them) and therefore they are absolutely continuous (in fact, each 
    $\gamma_{\nu}$ is clearly $\smoo^{\infty}$-smooth).
    We compute
    \begin{equation*}
	\forall \, \nu \in \posint, \, \forall \, t \in [0,R], \;
	\dkoba_M(\gamma_{\nu}(t),\gamma'_{\nu}(t)) =
	\dkoba_M\big(\phi_{\nu}(\sigma(t)),\sigma'(t)\phi'_{\nu}(\sigma(t))
	\big) \leq \dkoba_{\unitdisk}(\sigma(t),\sigma'(t))=1. 	 
	\end{equation*}
    \noindent (To write the second inequality above, we use the fact
    that $\phi_{\nu}$ is contractive relative to the Kobayashi metrics; the final 
    equality follows because $\sigma$ is a geodesic for the Poincar{\'e} metric.) 
    We also observe that
	\begin{equation*}
	\forall \, \nu \in \posint, \, \forall \, s,t \in [0,R], \;
	\koba_M(\gamma_{\nu}(s),\gamma_{\nu}(t)) = 
	\koba_M\big(\phi_{\nu}(\sigma(s)),\phi_{\nu}(\sigma(t))\big) 
	\leq \koba_{\unitdisk}(\sigma(s),\sigma(t)) = |s-t|.
	\end{equation*}
	\noindent (To write the second inequality above, we use the fact that 
	$\phi_{\nu}$ is contractive relative to the Kobayashi distances; the final
	equality follows because $\sigma$ is a geodesic for the Poincar{\'e} distance.)
	\smallskip
	
	Finally, 
	since $ |s - t | \leq R \leq \tanh^{-1}(\delta) \leq \tanh^{-1}(\tanh(\kappa)) = \kappa$ for all $ s, t \in [0, R] $, we have
	\[
	|s - t| -\kappa \leq 0 \leq \koba_M\big(\gamma_{\nu}(s), \gamma_{\nu}(t)\big) \leq |s - t| < |s - t| + \kappa 
	\]
	for all $ s, t \in [0, R] $. The above considerations show that
	each $\gamma_{\nu}$ is a $(1,\kappa)$-almost-geodesic.
	By our assumption that $M$ is a $ (1, \kappa) $-visibility submanifold, it follows
	that there exists a compact subset $K$ of $M$ such that, for every
	$\nu \in \posint$, $\mathsf{ran}(\gamma_{\nu}) \cap K \neq 
	\emptyset$. But $\mathsf{ran}(\gamma_{\nu}) \subset 
	\mathsf{ran}(\phi_{\nu})$ and, since $(\phi_{\nu})_{\nu \geq 1}$
	converges uniformly on the compact subsets of $\unitdisk$ to a 
	$\bdy M$-valued map, it follows that for every compact subset $K$ of
	$M$, $\mathsf{ran}(\gamma_{\nu}) \cap K = \emptyset$ for all $\nu$
	sufficiently large. This is a contradiction. So our starting
	assumption must be wrong, and $\kobMetGrM_M(r)\to 0$ as $r \to 0$.
\end{proof}

\begin{remark}
Theorem~\ref{th:visib_taut_KobMetGr_0} shows that, in particular, 
if a bounded, convex domain $\OM$ is a weak visibility domain, one has 
$\kobMetGrM_{\OM}(r)\to 0$ as $r\to 0$.
\end{remark}

We now prove the following analogue of Theorem~4.3 in \cite{Bharali_Maitra}.

\begin{theorem} \label{th:bdy-val_lims_cons}
	Suppose that $M$ is a bounded, connected, embedded complex submanifold of 
	$\C^d$. 
	Suppose that $M$ is a $(1,\kappa)$-visibility submanifold for some $\kappa>0$ 
	and that it is also taut. 
	Suppose that $X$ is a connected complex manifold and that 
	$(\phi_{\nu})_{\nu\geq 1}$ is a sequence of holomorphic maps from $X$
	to $M$ that converges, uniformly on the compact subsets of $X$, to a
	holomorphic map $\psi$ from $X$ to $\bdy M$. Then $\psi$ is constant.
\end{theorem}

\begin{proof}
        We argue exactly as in the proof of Theorem~4.3 in \cite{Bharali_Maitra}.
        The latter result was a direct consequence of the fact that 
        $\kobMetGrM_{\OM}(r)\to 0$ as $r\to 0$ for a taut domain $\OM$ satisfying 
        the visibility property. The corresponding result in the present case 
        is Theorem~\ref{th:visib_taut_KobMetGr_0}. 
\end{proof}

\subsection{Proof of Theorem~\ref{th:WD_submani}}\label{SS:proofwdt}
We are now ready to give a sketch of the proof of Theorem~\ref{th:WD_submani}.
\begin{proof}[Proof of Theorem~\ref{th:WD_submani}]
We argue exactly as in the proof of Theorem~1.8 in 
\cite{Bharali_Maitra}, replacing $\OM$ there by $M$, and consider two subcases as 
in the latter proof. In the first subcase, the results employed in the
latter proof are \cite[Theorem~4.3]{Bharali_Maitra} and 
\cite[Proposition~4.1]{Bharali_Maitra}. The results analogous to these in this 
paper are
Theorem~\ref{th:bdy-val_lims_cons} and Proposition~\ref{prp:tech_lem_vis_man}, 
respectively, which we employ in our argument to settle this subcase.
\smallskip

In the second subcase, the results employed in the proof of 
\cite[Theorem~1.8]{Bharali_Maitra} are \cite[Result~2.1]{Bharali_Maitra},
\cite[Lemma~2.9]{Bharali_Maitra} and
the existence of $(\lambda, \kappa)$-almost-geodesics on bounded domains. The 
results analogous to these in this paper are
Remark~\ref{Rm:analogKoba}, Theorem~\ref{T:exist_1kappa} and 
Result~\ref{Res:basic_but_important}, respectively, 
which we employ in our argument to settle this subcase and complete the proof.
\end{proof}

\noindent{\bf Acknowledgements.}
We thank the anonymous referee for several comments that helped to improve the exposition of the article. 
All three authors were supported by postdoctoral fellowships from the
Harish-Chandra Research Institute (Homi Bhabha National Institute) at the time of
doing this work.
\smallskip

\noindent{\bf Conflict of interest statement.} All the authors declare that there is no conflict of interest.

\end{document}